\documentclass[12pt, reqno]{amsart}

\usepackage{verbatim}
\usepackage{amssymb, amsmath, amsthm}
\usepackage{hyperref}
\usepackage{url} 
\usepackage[alphabetic,lite]{amsrefs}

\BibSpec{article}{%
    +{}  {\PrintAuthors}                {author}
    +{,} { \textit}                     {title}
    +{.} { }                            {part}
    +{:} { \textit}                     {subtitle}
    +{,} { \PrintContributions}         {contribution}
    +{.} { \PrintPartials}              {partial}
    +{,} { }                            {journal}
    +{}  { \textbf}                     {volume}
    +{}  { \PrintDatePV}                {date}
    +{,} { \issuetext}                  {number}
    +{,} { \eprintpages}                {pages}
    +{,} { }                            {status}
    +{,} { \url}                        {url}    
    +{,} { \PrintDOI}                   {doi}
    +{,} { available at \eprint}        {eprint}
    +{}  { \parenthesize}               {language}
    +{}  { \PrintTranslation}           {translation}
    +{;} { \PrintReprint}               {reprint}
    +{.} { }                            {note}
    +{.} {}                             {transition}
    +{}  {\SentenceSpace \PrintReviews} {review}
}

\BibSpec{book}{%
    +{}  {\PrintAuthors}                {author}
    +{,} { \textit}                     {title}
    +{.} { }                            {part}
    +{:} { \textit}                     {subtitle}
    +{,} { \PrintContributions}         {contribution}
    +{,} { \PrintConference}            {conference}
    +{}  {\PrintBook}                   {book}
    +{,} { }                            {booktitle}
    +{,} { }                            {publisher}
    +{,} { \PrintDateB}                 {date}
    +{,} { pp.~}                        {pages}
    +{,} { }                            {status}
    +{,} { \PrintDOI}                   {doi}
    +{,} { available at \eprint}        {eprint}
    +{}  { \parenthesize}               {language}
    +{}  { \PrintTranslation}           {translation}
    +{;} { \PrintReprint}               {reprint}
    +{.} { }                            {note}
    +{.} {}                             {transition}
}

\BibSpec{incollection}{%
    +{}  {\PrintAuthors}                {author}
    +{,} { \textit}                     {title}
    +{.} { }                            {part}
    +{:} { \textit}                     {subtitle}
    +{,} { \PrintContributions}         {contribution}
    +{,} { \PrintConference}            {conference}
    +{}  {\PrintBook}                   {book}
    +{,} { }                            {booktitle}
    +{,} { }                            {publisher}
    +{,} { \PrintDateB}                 {date}
    +{,} { pp.~}                        {pages}
    +{,} { }                            {status}
    +{,} { \PrintDOI}                   {doi}
    +{,} { available at \eprint}        {eprint}
    +{}  { \parenthesize}               {language}
    +{}  { \PrintTranslation}           {translation}
    +{;} { \PrintReprint}               {reprint}
    +{.} { }                            {note}
    +{.} {}                             {transition}
}

\BibSpec{misc}{%
    +{}  {\PrintAuthors}                {author}
    +{,} { \textit}                     {title}
    +{.} { }                            {part}
    +{:} { \textit}                     {subtitle}
    +{,} { \PrintContributions}         {contribution}
    +{.} { \PrintPartials}              {partial}
    +{,} { }                            {journal}
    +{}  { \textbf}                     {volume}
    +{}  { \PrintDatePV}                {date}
    +{,} { \issuetext}                  {number}
    +{,} { \eprintpages}                {pages}
    +{,} { }                            {status}
    +{,} { \url}                        {url}    
    +{,} { \PrintDOI}                   {doi}
    +{,} { available at \eprint}        {eprint}
    +{}  { \parenthesize}               {language}
    +{}  { \PrintTranslation}           {translation}
    +{;} { \PrintReprint}               {reprint}
    +{.} { }                            {note}
    +{.} {}                             {transition}
    +{}  {\SentenceSpace \PrintReviews} {review}
}

\usepackage{amscd}   
\usepackage{fullpage}
\usepackage[all]{xy} 
\usepackage{array} 
\usepackage[utf8]{inputenc}
\usepackage{pdflscape}

\usepackage[usenames, dvipsnames]{color}

\usepackage{graphicx}
\usepackage{tikz}
\usetikzlibrary{matrix,arrows,decorations.pathmorphing}

\usepackage{tikz-cd}

\usepackage{multirow}
\newcolumntype{M}[1]{>{\hbox to #1\bgroup\hss$}l<{$\egroup}}

\makeatletter
\newcommand\@brcolwidth{0.67em}

\def\@brarray[#1]{\array{r*\c@MaxMatrixCols {M{#1}}}}
\makeatother

\newtheorem{lemma}{Lemma}[section]
\newtheorem{theorem}[lemma]{Theorem}
\newtheorem{proposition}[lemma]{Proposition}

\newtheorem{cor}[lemma]{Corollary}

\newtheorem{claim*}{Claim}
\newtheorem{thm}[lemma]{Theorem}
\newtheorem*{thm*}{Theorem}
\newtheorem*{cor*}{Corollary}

\theoremstyle{definition}
\newtheorem{defn}[lemma]{Definition}
\newtheorem{conj}[lemma]{Conjecture}

\newtheorem{question}[lemma]{Question}
\newtheorem{remark}[lemma]{Remark}

\newtheorem{rmk}[lemma]{Remark}

\DeclareMathOperator{\Kum}{Kum}

\newcommand{\A}{{\bf A}}

\newcommand{\Q}{{\bf Q}}

\newcommand{\Z}{{\bf Z}}

\newcommand{\CC}{{\mathbb C}}
\newcommand{\GG}{{\mathbb G}}

\newcommand{\Xbar}{{\overline{X}}}

\newcommand{\Abar}{{\overline{A}}}

\newcommand{\kbar}{{\overline{k}}}

\newcommand{\Ybar}{{\overline{Y}}}

\newcommand{\calO}{{\mathcal O}}

\newcommand{\frakc}{{\mathfrak c}}

\newcommand{\frakf}{{\mathfrak f}}

\newcommand{\frakp}{{\mathfrak p}}
\newcommand{\frakq}{{\mathfrak q}}
\newcommand{\frakr}{{\mathfrak r}}
\newcommand{\fraks}{{\mathfrak s}}

\usepackage{mathrsfs} 
\newcommand{\scrA}{{\mathscr A}}

\newcommand{\scrK}{{\mathscr K}}

\newcommand{\scrO}{{\mathscr O}}

\DeclareMathOperator{\lcm}{lcm}

\DeclareMathOperator{\inv}{inv}

\DeclareMathOperator{\im}{im}

\DeclareMathOperator{\End}{End}

\DeclareMathOperator{\Hom}{Hom}

\DeclareMathOperator{\Aut}{Aut}
\DeclareMathOperator{\Gal}{{Gal}}

\DeclareMathOperator{\Cor}{Cor}
\DeclareMathOperator{\Res}{Res}

\DeclareMathOperator{\Br}{{Br}}

\DeclareMathOperator{\ord}{ord}
\DeclareMathOperator{\Sym}{Sym}

\DeclareMathOperator{\Pic}{Pic}

\DeclareMathOperator{\Spec}{{Spec}}

\DeclareMathOperator{\et}{\textrm{\normalfont \'et}}

\DeclareMathOperator{\BM}{BM}

\DeclareMathOperator{\rank}{rank}

\DeclareMathOperator{\GL}{GL}

\DeclareMathOperator{\disc}{disc}
\DeclareMathOperator{\id}{id}

\DeclareMathOperator{\NS}{NS}

\numberwithin{equation}{section}
\numberwithin{table}{section}

\usepackage{scalerel,stackengine}
\stackMath
\newcommand\reallywidehat[1]{%
\savestack{\tmpbox}{\stretchto{%
  \scaleto{%
    \scalerel*[\widthof{\ensuremath{#1}}]{\kern-.6pt\bigwedge\kern-.6pt}%
    {\rule[-\textheight/2]{1ex}{\textheight}}
  }{\textheight}%
}{0.5ex}}%
\stackon[1pt]{#1}{\tmpbox}%
}

\usepackage[margin=3.2cm]{geometry}

\title{Explicit uniform bounds for Brauer groups of singular K3 surfaces}
\author{Francesca Balestrieri}
\address{Francesca Balestrieri\\ American University of Paris\\ 5 Boulevard de La Tour-Maubourg\\
75007 Paris\\ France}
\email{fbalestrieri@aup.edu}

\author{Alexis Johnson}
\address{Alexis Johnson\\ Department of Mathematics\\ University of Minnesota\\ 206 Church St SE\\ Minneapolis, MN 55455\\ USA}
\email{akjohns@umn.edu}

\author{Rachel Newton}
\address{Rachel Newton\\
Department of Mathematics\\ King's College London\\ Strand\\ London WC2R 2LS\\
UK}
   \email{rachel.newton@kcl.ac.uk}



\begin{document}

\begin{abstract} 
Let $k$ be a number field.
We give an explicit bound, depending only on $[k:\Q]$ and the discriminant of the N\'{e}ron--Severi lattice, on the size of the Brauer group of a K3 surface $X/k$ that is geometrically isomorphic to the Kummer surface attached to a product of isogenous CM elliptic curves.
As an application, we show that the Brauer--Manin set for such a variety is effectively computable. Conditional on GRH, we can also make the explicit bound depend only on $[k:\Q]$ and remove the condition that the elliptic curves be isogenous.
In addition, we show how to obtain a bound, depending only on $[k:\Q]$, on the number of $\CC$-isomorphism classes of singular K3 surfaces defined over $k$, thus proving an effective version of the strong Shafarevich conjecture for singular K3 surfaces.
\end{abstract}    
    
\maketitle

\section{Introduction}
Let $k$ be a number field  with a fixed algebraic closure $\kbar$ and let $X$ be a smooth, projective, geometrically integral variety over $k$ with structure morphism $s : X \to \Spec k$. The Brauer group of $X$ is defined as $\Br X := \mathrm{H}^2_{\et}(X, \GG_m)$ and has a filtration
\[ \Br_0 X:= \im\left( \Br k\xrightarrow{s^\ast}\Br X\right) \subset \Br_1 X:= \ker\left(\Br X \to \Br\Xbar\right) \subset \Br X,\]
where $\Xbar:=X\times_k \kbar$.
In the 1970s, Manin proposed a systematic way to use the Brauer group to study the set  $X(k)$ of rational points of $X$, as follows (see \cite{Manin}). Consider the pairing 
\[\langle \ , \ \rangle_{\BM} : X(\A_k) \times \Br X \to \Q/\Z\]
given by $\langle (x_v)_v, \alpha \rangle_{\BM} := \sum_{v \in \Omega_k} \inv_v (x_v^\ast(\alpha))$ where, for each non-trivial place $v \in \Omega_k$, the map $\inv_v : \Br(k_v) \to \Q/\Z$ is the local invariant map coming from class field theory. Then it is easily seen that the closure $\overline{X(k)}$ of $X(k)$ in the adelic topology is contained in the left kernel of this pairing. We call this left kernel the \emph{Brauer--Manin set of $X$} and denote it by $X(\A_k)^{\Br}$. If $X$ \emph{satisfies the Hasse principle with Brauer--Manin obstruction}, meaning that $X(\A_k)^{\Br} = \emptyset$ if and only if $X(k) = \emptyset$, and if furthermore we have a way to \emph{effectively} compute the Brauer--Manin set $X(\A_k)^{\Br}$, then it follows that we can effectively decide whether $X$ has a rational point or not. Such effectivity results are related to Hilbert's famous tenth problem and its variations.

In this paper we focus on singular K3 surfaces and K3 surfaces that are geometrically Kummer surfaces of products of CM elliptic curves.  It is conjectured by Skorobogatov (see~\cite{Skorobogatov-K3Conj}) that, for any K3 surface $X$ over $k$, we have $\overline{X(k)} = X(\A_k)^{\Br}$. If this conjecture holds, then the problem of determining the qualitative arithmetic behaviour of the set of rational points of K3 surfaces is reduced to the problem of understanding their Brauer--Manin sets. A first step towards solving this problem is to study the relevant Brauer groups. By \cite[Theorem 1]{KT11},  it turns out that for effectivity problems concerning the computation of these Brauer--Manin sets, it suffices to effectively {bound} the size of $ \Br X/\Br_0 X$, which is finite for K3 surfaces (see \cite{SZ-K3finite}). Moreover, for K3 surfaces, V\'{a}rilly-Alvarado has postulated the existence of \emph{uniform} bounds for $\# (\Br X/\Br_0 X)$, although he makes no mention of effectivity of the bounds in the following conjecture:

\begin{conj}[Strong uniform boundedness {\cite[Conjecture 4.6]{VA-AWS}}]\label{conjtony} Fix a positive integer $n$ and a primitive
lattice $\Lambda \hookrightarrow \Lambda_{K3} := U^{\oplus 3} \oplus  E_8^{\oplus 2}$. Let $X$ be a K3 surface over a number field of degree $n$ such that $\NS\Xbar \cong  \Lambda$ as abstract lattices. Then there is a constant $C(n, \Lambda)$, independent of $X$, such that $ \#(\Br X/\Br_0 X) \leq  C(n, \Lambda)$.
\end{conj}
When $X$ is a K3 surface, explicit uniform bounds are already known for the size of $\Br_1 X/\Br_0 X$, see Remark~\ref{unifBr1}. Hence the real content of Conjecture~\ref{conjtony} is the existence of uniform bounds for the so-called transcendental part of the Brauer group, $\Br X/\Br_1 X$. Conjecture \ref{conjtony} can thus be viewed within the context of a more general question: 
\begin{question}[{\cite[Question 1.1]{VAV16}}]\label{qvav} Let $k$ be a number field. Let $Y$ be a smooth, projective surface over $k$ with trivial canonical sheaf. Is there a bound for $\# \im(\Br Y \to \Br \Ybar)$ that is independent of $Y$, depending only on, say, $h^1(Y, \scrO_Y)$, the geometric N\'{e}ron--Severi lattice $\NS \Ybar$, and $[k:\Q]$?
\end{question}

Our main aim in this paper is to give \emph{explicit uniform} bounds on the size of $\Br X/\Br_0 X$ in the case where $X/k$ is either a singular K3 surface or geometrically isomorphic to the Kummer surface associated to a product of CM elliptic curves. Following~\cite{Serre}, we write $M(n)$ for the smallest positive integer $N$ such that the order of any finite subgroup of $\GL_{n}(\Z)$ divides $N$. Minkowski gave a formula for $M(n)$ in~\cite{Minkowski}. Of particular relevance for our results is the constant $M(20)=2^{38}\cdot 3^{14}\cdot 5^6\cdot 7^3\cdot 11^2\cdot 13\cdot 17\cdot 19$. In our statements of the following results, we refer to the theorems in the body of the paper for more precise expressions.

\begin{thm*}[Corollary of Theorem \ref{thm:lattprodisog}] Let $k$ be a number field. Let $\Lambda$ be the N\'eron--Severi lattice of the Kummer surface of a product of isogenous (not necessarily full) CM elliptic curves over $\overline{k}$ and let $X/k$ be a K3 surface such that $\NS\overline{X}\cong\Lambda$ as abstract lattices.  Then
\begin{equation}\label{eq:uncondbnd}
 \# \frac{\Br X}{\Br_1 X} \leq 2^{34}\cdot 3^{3}\cdot\pi^{-2}\cdot   M(20)^4\cdot |\disc \Lambda|^2\cdot [k:\Q]^4.
 \end{equation}
\end{thm*}

\begin{remark}
The proof of Theorem~\ref{thm:lattprodisog} shows that the bound given in~\eqref{eq:uncondbnd} can be dramatically improved in special cases. For instance, if $\NS\overline{X}$ is generated by divisors that are defined over $k$ then the $M(20)^4$ factor can be eliminated from~\eqref{eq:uncondbnd}.
If $X$ is isomorphic over $k$ to the Kummer surface of a product of $k$-isogenous (not necessarily full) CM elliptic curves, then 
\[ \# \frac{\Br X}{\Br_1 X} \leq 2^{-2}\cdot \pi^{-2}\cdot |\Delta_K|^{-1}\cdot |\disc\Lambda|^2\cdot [k:\Q]^4\]
where $K=\Q(\sqrt{\disc\Lambda})$, and if, furthermore, the class number of $K$ is $1$ then 
\[ \# \frac{\Br X}{\Br_1 X} \leq 2^{-4}\cdot  |\Delta_K|^{-2}\cdot |\disc\Lambda|^2\cdot [k:\Q]^4.\]
\end{remark}

Shafarevich~\cite{Shafarevich} has conjectured that, for all $d\in\Z_{>0}$, there are
only finitely many lattices, up to isomorphism, which occur as the geometric N\'{e}ron--Severi lattice of a K3 surface defined over a number field of degree at most $d$. If this conjecture is true then the dependence on the lattice in Conjecture~\ref{conjtony} can be eliminated. In the CM setting, Orr and Skorobogatov proved the stronger statement (proved by Shafarevich for singular K3 surfaces) that there are only finitely many $\overline{\Q}$-isomorphism classes of K3 surfaces of CM type which can be defined over number fields of given degree~\cite[Theorem~B]{OrrSkoro}. However, their methods are not effective and, in particular, they do not enable us to eliminate the dependence on $\disc \Lambda$ in the bounds we describe above. Nevertheless, under the assumption of the Generalised Riemann Hypothesis, we can eliminate the dependence on the lattice $\Lambda$ and obtain the following result giving an explicit bound depending only on the degree $[k:\Q]$, at the expense of a larger power of the degree appearing in the bound.

\begin{thm*}[Theorem \ref{thm:unifsingK3_2}] 
Suppose that the Generalised Riemann Hypothesis holds. Let $k$ be a number field. Let $X/k$ be such that $\overline{X}$ is a Kummer surface with $\rank \NS\overline{X}=20$. 
Then there exists a finite extension $L/k$ such that $X_L\cong \Kum (E\times E')$ for some elliptic curves $E,E'$ over $L$ and we have 
\[ \# \frac{\Br X}{\Br_1 X} \leq (3.4)^2\cdot 10^{8}\cdot [L:\Q]^{12}\cdot\bigl((3.23)\cdot \log ([L:\Q]) +(2.73)\cdot 109\bigr)^4.\]
Moreover, we can choose $L$ such that $[L:k]\leq 2^9\cdot 3\cdot M(20)$.
\end{thm*}
For a generalisation of this result to singular K3 surfaces, see Theorem~\ref{thm:singcover}.
For an analogous result in the case where $X$ is geometrically isomorphic to the Kummer surface associated to a product of non-isogenous CM elliptic curves, see Theorem~\ref{unifKk2}.

\begin{rmk} \label{unifBr1}
For any K3 surface $X$ over a number field, $\Pic \Xbar$ is a free $\Z$-module of rank $r\leq 20$, whereby the proof of \cite[Lemma~6.4]{VAV16} shows that $\#(\Br_1 X / \Br_0 X)$ divides $M(r)^{r}$. Hence Theorems~\ref{thm:lattprodisog} and \ref{thm:unifsingK3_2} yield uniform bounds on the size of $\Br X / \Br_0 X$. This bound can be greatly improved in special cases -- for example, if $\Pic\overline{X}$ has a set of generators that are defined over $k$ then $\Br_1 X / \Br_0 X$ is trivial.
\end{rmk}

Using their proof of Shafarevich's conjecture for K3 surfaces of CM type, Orr and Skorobogatov proved Conjecture~\ref{conjtony} for K3 surfaces of CM type by showing the existence of a bound depending only on the degree $[k:\Q]$, see \cite[Corollary~C.1]{OrrSkoro}. However, it is not clear how to make their bound effective, let alone explicit. The value of our results lies in their
explicit nature, which allows us to obtain the following effectivity result.
\begin{thm*}[Theorem \ref{thm:effcomp}]
Let $k$ be a number field and let $X/k$ be such that $\overline{X}$ is a Kummer surface with $\rank \NS\overline{X}=20$. 
Then $X(\A_k)^{\Br}$ is effectively computable. 
\end{thm*}

Under the assumption of the Generalised Riemann Hypothesis, we obtain a similar effectivity result for a wider class of K3 surfaces -- see Theorem~\ref{thm:effcompGRH}.
It is important to note that in results like Theorems~\ref{thm:lattprodisog} and~\ref{thm:effcompGRH} we allow complex multiplication by orders other than the full ring of integers of the CM field. In particular, our objects of study include varieties not tackled by Valloni in \cite[\S11]{Valloni}, where he gave an effective algorithm which computes bounds on the size of $\Br \Xbar^{\Gal(\kbar/k)}$  (and consequently on $\Br X/\Br_1 X$ and $\Br X/\Br_0 X$) in the case where $X/k$ is a \emph{principal} CM K3 surface. For examples of non-principal CM Kummer surfaces attached to products of CM elliptic curves, see \cite[Example 9.8]{Ito} for some cases where the elliptic curves are not isogenous, and \cite{Laface} and \cite[Theorem 3.2]{Valloni2} for some cases where the elliptic curves are isogenous. It would be interesting to investigate whether Valloni's work can be used to obtain explicit uniform bounds for the transcendental parts of Brauer groups of principal CM K3 surfaces. 

Our results for Kummer surfaces make use of the close relationship between the transcendental parts of the Brauer groups of abelian surfaces and the associated Kummer surfaces (see \cite{SZ12}). One of the inspirations for our work was the paper \cite{VAV16} of V\'{a}rilly-Alvarado and Viray, in which they studied Question~\ref{qvav} for abelian surfaces and Kummer surfaces related to products of isogenous non-CM elliptic curves. In this context, they showed that the existence of uniform bounds (depending only on the degree of the base number field) on the odd order transcendental parts of the relevant Brauer groups is equivalent to the existence of a uniform bound on the odd parts of integers $n$ for which
there exist non-CM elliptic curves with abelian $n$-division fields. For a fixed prime $\ell$, they gave uniform bounds on the $\ell$-primary subgroups of the transcendental parts of the Brauer groups. In~\cite{CFTTV16Published} Cantoral-Farf\'{a}n, Tang, Tanimoto and Visse gave effective bounds for Brauer groups of Kummer surfaces associated to Jacobians of genus $2$ curves over number fields. Their bounds depend on the Faltings height as well as the degree of the base field. For abelian varieties of arbitrary dimension, Gaudron and R\'{e}mond obtained bounds on the transcendental part of the Brauer group depending on the dimension, the Faltings height and the degree of the base field \cite{GR20}. 

Our next result is an example of the kind of bound we obtain in the abelian setting, depending only on the degree of the base field and the geometric N\'{e}ron--Severi lattice.

\begin{thm*}[Corollary of Theorem~\ref{thm:ablatt}]
Let $k$ be a number field and let $A/k$ be an abelian surface such that $\NS \overline{A}$ contains a hyperbolic plane and $\rank \NS \overline{A} = 4$.
Then
\[ \# \frac{\Br A}{\Br_1 A} \leq 2^{18}\cdot 3^3\cdot\pi^{-2}\cdot |\disc \Lambda|^2 \cdot [k:\Q]^4 .\]
\end{thm*}

Conditional on the Generalised Riemann Hypothesis, we obtain the following uniform bound, which depends only on the degree of the base field:

\begin{thm*}[Theorem \ref{thmabvar}] 
Suppose that the Generalised Riemann Hypothesis holds. Let $k$ be a number field and let $A/k$ be an abelian surface such that $\NS \overline{A}$ contains a hyperbolic plane and $\rank \NS \overline{A} = 4$.
 Let $L/k$ be a finite extension such that $\End (A_L)=\End (\overline{A})$. Then
 \[ \# \frac{\Br A}{\Br_1 A} \leq(3.4)^2\cdot 10^{8}\cdot [L:\Q]^{12}\cdot\bigl((3.23)\cdot \log ([L:\Q]) +(2.73)\cdot 109\bigr)^4.\]
 Moreover, we can choose $L$ such that $[L:k]\leq 2^4\cdot 3$. 
\end{thm*}
  
In the course of our work, we obtain bounds for the conductors of endomorphism rings of CM elliptic curves over number fields, which may be of independent interest.

\begin{thm*} [Corollary \ref{cor:conductor}] Let $k$ be a number field and let $E/k$ be an elliptic curve with CM by an order of conductor $\mathfrak{f}$ in an imaginary quadratic field.  Then 
\[\frakf\leq \min\left\{3 \cdot [k:\Q]^2,   \max\{[k:\Q]^2,7\} \right\}.\]
\end{thm*}

We use this and similar results to obtain bounds on the transcendental parts of Brauer groups related to products of CM elliptic curves. Our bounds on conductors also yield an explicit version of the main result of \cite{Shafarevich}:

\begin{thm*}[Corollary~\ref{strongSha}] 
The number of $\CC$-isomorphism classes of singular K3 surfaces defined over number fields of degree at most $d$ is bounded above by 
\[\begin{array}{l}
3\cdot M(20)^3\cdot d^3\cdot(\log(3\cdot M(20)^2\cdot d^2)+1)\\ 
 \cdot\# \{K \textup{ imaginary quadratic}\mid h_K\leq M(20)\cdot d\}.\end{array}\]
\end{thm*}

\begin{remark} 
\begin{enumerate}
\item The quantity 
$\# \{K \textup{ imaginary quadratic}\mid h_K\leq d\}$ can be bounded explicitly  in terms of $d$ using the Siegel--Tatuzawa Theorem~\cite{Tatuzawa}.

\item In \cite[Theorem~1]{Sound}, Soundararajan shows that for $d$ sufficiently large
\[\# \{K \textup{ imaginary quadratic}\mid h_K\leq d\}=\frac{3\zeta(2)}{\zeta(3)}\cdot d^2+O_{\epsilon}(d^2\cdot( \log d)^{-\frac{1}{2}+\epsilon}).\]
\end{enumerate}
\end{remark}

\subsection{Notation and terminology}   Throughout this paper, we use the following notation:  
\[
\begin{array}{rl}
k & \textrm{is a field of characteristic 0},\\
\kbar & \textrm{is a fixed algebraic closure of $k$},\\
\Gamma_k & \textrm{is the absolute Galois group $\Gal(\kbar/k)$ of $k$},\\
\Omega_k & \textrm{is the set of non-trivial $k$-places, when $k$ is a number field},\\
X & \textrm{is a variety over $k$},\\
X_l & \textrm{is the base-change $ X \times_{\Spec k} \Spec l$ of $X$ to 
 $l/k$},\\
\overline{X}& \textrm{denotes $X_{\overline{k}}$},\\
\Br_1 X & \textrm{denotes $ \ker(\Br X \to \Br \Xbar)$},\\
\Br_0 X & \textrm{denotes $  \im (\Br k \to \Br X)$},\\
 \Br_1 X/\Br_0 X & \textrm{is the \emph{algebraic} part of the Brauer group of $X$},\\
  \Br X/\Br_1 X & \textrm{is the \emph{transcendental} part of the Brauer group of $X$}.\\
\end{array}\]

For an abelian group scheme $A$ over $k$  and an integer $d\in\Z_{>0}$, we use the following notation:  
\[
\begin{array}{rl}
A[d] & \textrm{denotes the $d$-torsion subgroup of $A(\kbar)$},\\
A\{d\}& \textrm{denotes the $d$-primary part $\varinjlim\limits_{n} A [d^n]$ of $A(\kbar)$}.
 \end{array}\]
For an elliptic curve $E$ defined over $k$, we use the following notation:
\[
\begin{array}{rl}
\End(\overline{E}) & \textrm{denotes the full ring of endomorphisms defined over $\overline{k}$},\\ 
\End_k(E) & \textrm{denotes the subring of endomorphisms defined over $k$.}
 \end{array}\]
 We say that $E/k$ has complex multiplication (CM) by an order $\calO$ in an imaginary quadratic field if $\End(\overline{E})=\calO$. We say that $E/k$ has full CM if $\End(\overline{E})$ is isomorphic to the ring of integers of an imaginary quadratic field.

\smallskip

For an imaginary quadratic field $K$,  we use the following notation:
\[
\begin{array}{rl}
\Delta_K & \textrm{denotes the discriminant of $K$},\\
h_K& \textrm{denotes the class number of $K$},\\
\calO_K & \textrm{denotes the ring of integers of $K$},\\
\calO_{K,\frakf} & \textrm{denotes the order of conductor $\frakf$ inside $\calO_K$},\\
\calO_{\frakf} & \textrm{denotes the order of conductor $\frakf$ inside $\calO_K$ when $K$ is clear
},\\
K_\frakf & \textrm{denotes the ring class field associated to the order $\calO_{K,\frakf}$},
 \end{array}\]
and for an order $\mathcal{O}$ in $K$:
\[
\begin{array}{rl}
h(\mathcal{O}) & \textrm{denotes the class number of $\mathcal{O}$.}
\end{array}\]
Throughout the paper we fix embeddings $k\hookrightarrow \overline{k}\hookrightarrow \CC$ and consider all field extensions of $k$ of finite degree as subfields of $\kbar$.

\subsection{Acknowledgements}
The authors are very grateful to \'{E}ric Gaudron and Ga\"{e}l R\'{e}mond for pointing out some errors in a previous draft of this article and providing some useful references. They are also indebted to the anonymous referee whose helpful comments improved the article and its exposition.
The authors thank Tim Browning, Jennifer Berg, Titus Hilberdink, Adam Logan, Jack Petok, Matthias Sch\"{u}tt, Alexei Skorobogatov, Domenico Valloni, Tony V\'{a}rilly-Alvarado and Bianca Viray for useful discussions.  Francesca Balestrieri was partially supported by the European Union's Horizon 2020
research and innovation programme under the Marie Sk\l{}odowska-Curie grant 840684. Rachel Newton was supported by EPSRC grant EP/S004696/1 and UKRI Future Leaders Fellowship MR/T041609/1.


\section{Abelian surfaces of product type}\label{sec:ab}
\begin{defn} Let $k$ be a field of characteristic 0. 
Denote by $\scrA_k$ the set of abelian surfaces $A/k$ such that $ \NS \overline{A}$ contains a hyperbolic plane. For a lattice $\Lambda$ containing a hyperbolic plane, denote by $\scrA_{k,\Lambda}$ the set of abelian surfaces $A/k$ such that $ \NS \overline{A}$ is isomorphic to $\Lambda$ (as an abstract lattice, with no Galois action).
\end{defn}

The lemma below shows that $\scrA_k$ consists of the surfaces that are geometrically isomorphic to products of elliptic curves.

\begin{lemma}[{\cite[Lemma 2.8]{VAV16}}] \label{lem:NSrks} Let $A$ be an abelian surface over an algebraically closed field such that $\NS A$
contains a hyperbolic plane. Then $A$ is isomorphic to a product of elliptic curves. In addition,
\begin{itemize}
\item if $\rank \NS A = 2$, then the elliptic curves are not isogenous,
\item if $\rank\NS A = 3$, then the elliptic curves are non-CM, isogenous, and the degree of
a cyclic isogeny between them is $\frac{1}{2}\disc \NS A$, and
\item if $\rank\NS A = 4$, then the elliptic curves are isogenous and CM.
\end{itemize}
\end{lemma}

Next, we bound the degree of a number field over which an element of $\scrA_k$ becomes isomorphic to a product of elliptic curves.

\begin{proposition}\label{prop:productsurface}
Let $k$ be a field of characteristic 0 and let $A\in\scrA_k$. Then there exist a finite extension $L/k$ with $[L:k]\leq 2^4\cdot 3$ and 
elliptic curves $E$ and $E'$ over $L$ such that 
\[ A_L\cong E \times E'.\] 
Furthermore, if $\rank \NS\overline{A}=2$, then $[L:k]\leq 2$. If $\rank \NS\overline{A}=3$, then $[L : k] \in \{1, 2, 3, 4, 6, 8, 12\}$. If $ \NS \overline{A}$ is isomorphic (as an abstract lattice) to the N\'eron--Severi lattice of a product of isogenous elliptic curves with CM by $K$, then $K\subset L$, the elliptic curves $E$ and $E'$ have CM by $K$, and $\End (A_L)=\End(\overline{A})$.
\end{proposition}

\begin{proof}
By Lemma~\ref{lem:NSrks}, there exist elliptic curves $E$ and
$E'$ over $\overline{k}$ such that $\overline{A}\cong E \times E'$. By viewing the projections onto $E$ and $E'$ as endomorphisms of $\overline{A}$, one sees that if for some field extension $L/k$ we have $\End( A_L)=\End(\overline{A})$ then it follows that $A_L$ is isomorphic to a product of elliptic curves over $L$. As a consequence of~\cite[Theorem~4.3]{FKRS}, there exists a Galois extension $L/k$ of degree at most $2^4\cdot 3$ such that $\End (A_L)=\End(\overline{A})$. (See also~\cite[Th\'{e}or\`{e}me~1.1]{Remond} for a result for abelian varieties of arbitrary dimension.) If $ \NS \overline{A}$ is isomorphic (as an abstract lattice) to the N\'eron--Severi lattice of a product of isogenous elliptic curves with CM by $K$ then $ \End( A_L)\otimes \Q\cong M_2(K)$ and, in particular, $K$ is fixed by $\Gal(\overline{k}/L)$.

For the cases where $\rank \NS\overline{A}\leq 3$, we follow the proof of \cite[Proposition 2.7]{VAV16}. 
Let $\mathcal{M}_{1,1} := \mathbb{A}^1$ denote the coarse moduli space of elliptic curves, parametrised by the $j$-invariant.
The coarse moduli space $\mathcal{A}_2$ of principally polarised abelian surfaces contains the
Humbert surface $\mathcal{H}_1 := \Sym^2 \mathcal{M}_{1,1} $, which is the locus of abelian surfaces with product structure. We have an isomorphism $\Sym^2 \mathcal{M}_{1,1}\cong \mathbb{A}^2$ given by sending the class of $(j_1,j_2)$ to $(j_1+j_2,j_1\cdot j_2)$.

Since $\overline{A}\cong E \times E'$, the surface $A$ gives rise to a point $x \in \mathcal{H}_1(\overline{k})$, which has
coordinates $(j(E) + j(E'), j(E)\cdot j(E'))$ when viewed as a point in $\mathbb{A}^2(\overline{k})$.
For any $\sigma \in \Gamma_k$, we have $\sigma(A) = A$,
so $E \times E'\cong \sigma (E \times E'
)$, and thus $x \in \mathcal{H}_1(k)$. Therefore $j(E) + j(E')$ and $j(E)\cdot j(E')$ belong
to $k$, and so there is an extension $k_0/k$ of degree at most $2$ such that $ j(E), j(E') \in k_0$.
Therefore, we may assume that $E$ and $E'$ are defined over $k_0$. Now $A_{k_0}$
is a twist (as an abelian surface) of $E \times E'$ and hence corresponds to an element of
$\mathrm{H}^1(k_0, \Aut(\overline{E} \times \overline{E'}))$. Let $L/k_0$ be a field extension. 
The abelian surface $A_{L}$
is isomorphic to a product of elliptic curves over $L$ if
\[ [A_L]\in \im (\mathrm{H}^1(L, \Aut(\overline{E})) \times \mathrm{H}^1(L, \Aut(\overline{E'}))\to\mathrm{H}^1(L, \Aut(\overline{E}\times \overline{E'})) ).\]

\medskip

\paragraph{\textbf{Case 1: $\rank \NS\overline{A}=2$}} In this case $\Aut(\overline{E}\times \overline{E'})=\Aut(\overline{E})\oplus \Aut(\overline{E'})$ and hence $\mathrm{H}^1(k_0, \Aut(\overline{E})) \times \mathrm{H}^1(k_0, \Aut(\overline{E'}))\to\mathrm{H}^1(k_0, \Aut(\overline{E}\times \overline{E'})) $ is an isomorphism. Therefore, $A_{k_0}$ is isomorphic to a product of elliptic curves over $k_0$.

\medskip

\paragraph{\textbf{Case 2: $\rank \NS\overline{A}=3$}} This is the case treated in \cite[Proposition 2.7]{VAV16}. The authors show that for any $\phi\in \mathrm{H}^1(k_0, \Aut(\overline{E} \times \overline{E'}))$ there exists a field extension $L/k_0$ with $[L:k_0]\in\{1,2,3,4,6\}$ such that
\[ \Res_{L/k_0}\phi\in \im (\mathrm{H}^1(L, \Aut(\overline{E})) \times \mathrm{H}^1(L, \Aut(\overline{E'}))\to\mathrm{H}^1(L, \Aut(\overline{E}\times \overline{E'})) ).\]
Observing that $[L:k]=[L:k_0]\cdot [k_0:k]$ and $[k_0:k]\leq 2$ yields the result.
\end{proof}

The following result is surely well known but we include it here for completeness since we will use it later.
\begin{lemma}\label{lem:allisog}
Let $E$ and $E'$ be elliptic curves defined over a field $k$ of characteristic 0 such that $\End(\overline{E})\otimes\mathbb{Q}\subset k$. Suppose that there exists a $k$-isogeny $\phi:E\to E'$. Then all isogenies between $\overline{E}$ and $\overline{E'}$ are induced by isogenies defined over $k$.
\end{lemma}

\begin{proof}
Let $\overline{\phi}:\overline{E}\to\overline{E'}$ denote the induced isogeny and let $\overline{\phi}^\vee$ denote its dual. Then $\psi\mapsto \overline{\phi}^\vee\circ\psi$ gives an injective map of Galois modules $\Hom(\overline{E},\overline{E'})\to \End(\overline{E})$. Since $\End(\overline{E})\otimes\mathbb{Q}\subset k$, the action of $\Gamma_k$ on $\End(\overline{E})$ is trivial and hence all elements of $\Hom(\overline{E},\overline{E'})$ are fixed by $\Gamma_k$, as required.
\end{proof}

The final results in this section show how to read information about the CM orders of isogenous elliptic curves $E_1$ and $E_2$ from the N\'{e}ron--Severi lattice of their product.

\begin{proposition}[{\cite[Corollary~24]{Kani2016}}]\label{prop:Kani16}
Let $E_1$ and $E_2$ be isogenous elliptic curves over an algebraically closed field of characteristic 0. Then
\begin{equation}\label{eq:Kani16}
\disc \NS(E_1\times E_2)=-(-2)^{\rho-2}\cdot\disc\Hom(E_1,E_2),
\end{equation}
where $\rho:=\rank\NS(E_1\times E_2)$ and $\Hom(E_1,E_2)$ is a lattice with pairing $\langle\varphi,\psi\rangle:=\frac{1}{2}(\deg(\varphi+\psi)-\deg\varphi-\deg\psi)$.
\end{proposition}

Note that Lemma~\ref{lem:NSrks} shows that $\rho-2=\rank\Hom(E_1,E_2)$. In \cite{Kani2016}, Kani considers the pairing $\langle\varphi,\psi\rangle:=\deg(\varphi+\psi)-\deg\varphi-\deg\psi$ on $\Hom(E_1,E_2)$, whence the power of $2$ in \eqref{eq:Kani16}.

\begin{proposition}[{\cite[Corollary~42]{Kani2011}}]\label{prop:Kani11}
Let $E_1$ and $E_2$ be elliptic curves over an algebraically closed field of characteristic 0 with CM by orders with conductors $\mathfrak{f}_1$ and $\mathfrak{f}_2$, respectively, in an imaginary quadratic field $K$. Then
\[\disc\Hom(E_1,E_2)=-2^{-2}\cdot\lcm(\mathfrak{f}_1,\mathfrak{f}_2)^2\cdot\Delta_K\]
where $\Hom(E_1,E_2)$ is a lattice with pairing $\langle\varphi,\psi\rangle:=\frac{1}{2}(\deg(\varphi+\psi)-\deg\varphi-\deg\psi)$. 
\end{proposition}

\begin{proof}
This follows immediately from \cite[Corollary~42]{Kani2011} upon noting that the author's $\Delta(q_{E_1,E_2})$ is equal to $-4 \cdot \disc\Hom(E_1,E_2)$.
\end{proof}

Combining Propositions~\ref{prop:Kani16} and \ref{prop:Kani11} gives the following corollary.

\begin{cor}\label{cor:NSExE}
Let $E_1$ and $E_2$ be elliptic curves over an algebraically closed field of characteristic 0 with CM by orders with conductors $\mathfrak{f}_1$ and $\mathfrak{f}_2$, respectively, in an imaginary quadratic field $K$. Then
\[\disc \NS(E_1\times E_2)=\lcm(\mathfrak{f}_1,\mathfrak{f}_2)^2 \cdot\Delta_K.\]
\end{cor}


\section{The associated Kummer surfaces}\label{sec:Kummer}

\begin{defn}[{\cite[Definition 2.1]{SZ-KumVar}}] Let $k$ be a field of characteristic 0. Let $A$ be an abelian surface over $k$. Any $k$-torsor $T$ under $A[2$]
gives rise to a $2$-covering $\rho: V \to A$, where $V$ is the quotient of $A\times_k T$ by the diagonal
action of $A[2]$ and $\rho$ comes from the projection onto the first factor. Then $T =\rho^{-1}(O_A)$ and $V$ has
the structure of a $k$-torsor under $A$. The class of $T$ maps to the class of $V$ under the
map $\mathrm{H}^1_{\et}(k, A[2]) \to \mathrm{H}^1_{\et}(k, A)$ induced by the inclusion of group schemes $A[2] \to A$ and, in
particular, the period of $V$ divides $2$. 
Let $\sigma : \tilde{V} \to V$ be the blow-up of $V$ at $T \subset V$. The involution $[-1]: A \to A$ fixes $A[2]$
and induces involutions $\iota$ on $V$ and $\tilde{\iota}$ on $\tilde{V}$ whose fixed point sets are $T$ and the exceptional
divisor, respectively. We call $\Kum V  :=\tilde{V}/\tilde{\iota}$ the \emph{Kummer surface associated
to $V$ (or $T$)}. We remark that the quotient $\Kum V $ is geometrically isomorphic to 
$\Kum A$, so in particular it is smooth.
\end{defn}

\begin{defn}
For a lattice $\Lambda$, denote by $\scrK_{k,\Lambda}$ the set of smooth, projective K3 surfaces $X/k$ such that $ \NS \overline{X}$ is isomorphic to $\Lambda$ (as an abstract lattice, with no Galois action). Let $S$ be the set of lattices that occur as the N\'eron--Severi lattice of the Kummer surface of a product of elliptic curves over $\overline{k}$, and let $\scrK_k:=\bigcup_{\Lambda\in S}\scrK_{k,\Lambda}$.
\end{defn}

\begin{defn} Let $X := \Kum Y$ be a Kummer surface over $k$, where $Y \to A$ is a 2-covering of some abelian surface $A$ over $k$.
Consider the natural blow-up map $\Xbar \to \Ybar/\iota_{\kbar}$, where $\iota_{\kbar}: \Abar \to \Abar$ is the antipodal involution, whose exceptional divisor consists of 16 pairwise disjoint smooth rational $(-2)$-curves and forms a  sublattice $\Z^{16} \subset \NS \Xbar$.  The \emph{Kummer lattice associated to $\Xbar$}, denoted by $\Lambda_K$,  is the saturation of this sublattice. It can be shown that $\Lambda_K$ is an even, negative-definite lattice of rank 16 and discriminant $2^6$ whose isomorphism type is independent of the choice of $Y$. (For more details about the Kummer lattice, we refer the reader to e.g.~\cite{LP}.)
\end{defn}

The next results allow us to bound the degree of a field extension over which an element of $\scrK_{k}$ becomes the Kummer surface attached to a product of elliptic curves.

\begin{proposition}[{\cite[Proposition~2.1]{VAV16}}]\label{prop:Nik}
There is a positive integer $N$ such that for any number field $k$, and any K3 surface $X/k$ with $\NS \overline{X}$ containing a sublattice isomorphic to $\Lambda_K$, there is an extension $k_0/k$ of degree at most $N$ such that $X_{k_0}$ is a Kummer surface.
\end{proposition}

\begin{theorem}\label{thm:KumE}
Let $X := \Kum Y$ be a Kummer surface over a field $k$ of characteristic $0$, where $Y \to A$ is a $2$-covering  of some abelian surface $A$ over $k$. Assume that $X\in\scrK_k$.
Then there exist a field extension $L/k$ with $[L:k]\leq 2^4\cdot 3$ and elliptic curves $E$ and $E'$ over
$L$ such that
 \[A_L\cong E \times E'.\]
 Furthermore, if $\rank \NS\overline{X}=18$, then $[L:k]\leq 2$. If $\rank \NS\overline{X}=19$, then $[L : k] \in \{1, 2, 3, 4, 6, 8, 12\}$. If $ \NS \overline{A}$ is isomorphic (as an abstract lattice) to the N\'eron--Severi lattice of a product of isogenous elliptic curves with CM by $K$, then $K\subset L$, the elliptic curves $E$ and $E'$ have CM by $K$, and $\End (A_L)=\End(\overline{A})$.
\end{theorem}

\begin{proof}
There is an exact sequence of lattices
\[0 \to \Lambda_K \to \NS \overline{X} \to \NS \overline{Y} \to 0,\]
where $\Lambda_K$ is the Kummer lattice and the map $\Lambda_K \to \NS \overline{X}$ is the natural inclusion (see \cite[Remark 2]{SZ12}, for example). Since $X\in\scrK_k$, this implies that $\NS\overline{Y}$ is isomorphic as an abstract lattice to $\NS(\overline{E}\times \overline{E'})$ for some elliptic curves $\overline{E}$ and $\overline{E'}$ defined over $\overline{k}$. Since $\overline{Y}\cong \overline{A}$, this shows that $A\in\scrA_k$.
Now apply Proposition~\ref{prop:productsurface}.
\end{proof}

\begin{cor} \label{cor:KumE} There exists a positive integer $M_0$ such that, for all number fields $k$ and all surfaces $X \in \scrK_k$, there exist: a field extension $L_0/k$ of degree at most $M_0$, elliptic curves $E$ and $E'$ over $L_0$, and a $2$-covering $Y \to E \times E'$ such that
\[X_{L_0}\cong \Kum Y.\]
\end{cor}

\begin{proof}
This follows immediately from Proposition~\ref{prop:Nik} and Theorem~\ref{thm:KumE}.
\end{proof}

\begin{remark}\label{rem:M0} The proof of \cite[Proposition 2.1]{VAV16} shows that one may take $N=2 \cdot M(20)$ in Proposition~\ref{prop:Nik}. Therefore, by Theorem \ref{thm:KumE} one may take $M_0=2^5 \cdot 3 \cdot M(20)$ in Corollary \ref{cor:KumE}. With more information about the K3 surface $X$, this bound can be improved.
 If, for example, $X \in \scrK_k$ satisfies $\rank\NS\overline{X}=18$ then one may take $N=2\cdot M(18)$ and $M_0=2^2\cdot M(18)$.
\end{remark}

\begin{proposition}[] \label{propVAV}
Let $A$ be an abelian surface over a field $k$ of characteristic $0$ and let
$Y\to A$ be a $2$-covering. Then there exists a field extension $L_1/k$ with $[L_1:k] \leq 2^4$ such that $Y_{L_1}\cong A_{L_1}$.
\end{proposition}

\begin{proof}
Since $f:Y\to A$ is a $2$-covering, there exists a field extension $L_1/k$ with $[L_1:k]\leq\# A[2]=2^4$ such that $Y_{L_1}\cong A_{L_1}$.
\end{proof}

\begin{rmk}
In Proposition~\ref{propVAV}, if $Y\to A$ is the trivial 2-covering then we can take $L_1=k$.
\end{rmk}

\begin{cor}\label{cor:Kumab}Let $E$ and $E'$ be elliptic curves over a field $k$ of characteristic $0$, let $Y\to E\times E'$ be a $2$-covering and let $X:=\Kum Y$. Then there exists a field extension $L_1/k$ with $[L_1:k] \leq 2^4$ such that, for all $n \in \Z_{>0}$,
\begin{equation}\label{eq:Brn}
 \frac{ \Br X_{L_1}[n]}{\Br_1 X_{L_1}[n]} \hookrightarrow \frac{ \Br (E_{L_1} \times E'_{L_1})[n]}{\Br_1 (E_{L_1} \times E'_{L_1})[n]},
 \end{equation}
and hence
\begin{equation}\label{eq:Brinj}
 \frac{ \Br X_{L_1}}{\Br_1 X_{L_1}} \hookrightarrow \frac{ \Br (E_{L_1} \times E'_{L_1})}{\Br_1 (E_{L_1} \times E'_{L_1})}.
 \end{equation}
\end{cor}

\begin{proof}
By Proposition~\ref{propVAV} there exists a field extension $L_1/k$ with $[L_1:k] \leq 2^4$ such that $Y_{L_1}\cong E_{L_1} \times E'_{L_1}$ and therefore $X_{L_1}\cong \Kum(E_{L_1} \times E'_{L_1})$.
By \cite[Theorem~2.4]{SZ12}, we have an injection \[\frac{ \Br X_{L_1}[n]}{\Br_1 X_{L_1}[n]} \hookrightarrow \frac{ \Br (E_{L_1} \times E'_{L_1})[n]}{\Br_1 (E_{L_1} \times E'_{L_1})[n]} \] which is an isomorphism if $n$ is odd. 
The statement~\eqref{eq:Brinj} follows from~\eqref{eq:Brn} since the Brauer groups in question are torsion by \cite[Proposition 1.4]{Grothendieck}.
\end{proof}

\begin{rmk} For $n$ odd, \eqref{eq:Brn} holds with $L_1=k$ and, furthermore, the injection in  \eqref{eq:Brn} is an isomorphism. Indeed, if $n$ is odd, apply \cite[Proposition~3.3]{VAV16} to the $2$-covering $f:Y\to E\times E'$ to get an isomorphism $f^*:  \frac{ \Br (E\times E')[n]}{\Br_1 (E\times E')[n]}\to \frac{\Br Y[n]}{\Br_1 Y[n]} $. 
\end{rmk}

The next two results show how to obtain information about the CM orders of isogenous elliptic curves $E_1$ and $E_2$ from the N\'{e}ron--Severi lattice of $\Kum(E_1\times E_2)$.

\begin{theorem}[{\cite[Theorem~3.3]{Shioda}}]\label{thm:Shioda}
Let $E_1$ and $E_2$ be elliptic curves over an algebraically closed field of characteristic 0. Then \[|\disc \NS(\Kum(E_1\times E_2))|=2^4\cdot|\disc\Hom(E_1,E_2)|,\]
where $\Hom(E_1,E_2)$ is a lattice with pairing $\langle\varphi,\psi\rangle:=\frac{1}{2}(\deg(\varphi+\psi)-\deg\varphi-\deg\psi)$.
\end{theorem}

\begin{cor}\label{cor:NSKum}
Let $E_1$ and $E_2$ be elliptic curves over an algebraically closed field $k$ of characteristic 0 with CM by orders with conductors $\mathfrak{f}_1$ and $\mathfrak{f}_2$, respectively, in an imaginary quadratic field $K$. Then
 \[|\disc \NS(\Kum(E_1\times E_2))|=2^2\cdot\lcm(\mathfrak{f}_1,\mathfrak{f}_2)^2\cdot|\Delta_K|.\]
\end{cor}

\begin{proof}
This is an immediate consequence of Theorem~\ref{thm:Shioda} and Proposition~\ref{prop:Kani11}.
\end{proof}

\section{Bounds on conductors and the Shafarevich conjecture for singular K3 surfaces}
The main results of this section are Theorem~\ref{strongSha1} and Corollary~\ref{strongSha} which yield an explicit version of Shafarevich's finiteness result for $\CC$-isomorphism classes of singular K3 surfaces defined over a number field.
We begin with an auxiliary result giving bounds on conductors of orders in CM fields. These bounds are used in Corollary~\ref{cor:EC_CM_K} and Proposition~\ref{prop:ECMbound} to bound the number of $\CC$-isomorphism classes of elliptic curves defined over number fields of bounded degree. They are also used in the proof of Theorem~\ref{strongSha1} and they appear again in Section~\ref{sec:ExE} where they are used to obtain bounds on the transcendental part of the Brauer group of a self-product of a CM elliptic curve.

\begin{proposition}\label{prop:aux_cond}
Let $K$ be an imaginary quadratic field and let $\calO_{\mathfrak{f}}$ be an order of conductor $\mathfrak{f}>0$ in $\calO_K$. Let $K_\mathfrak{f}$ denote the ring class field associated to $\calO_\mathfrak{f}$.  Then
\begin{enumerate}
\item \label{Q7} if $K=\Q(\sqrt{-7})$, we have $\mathfrak{f}\leq \max\{[K_{\mathfrak{f}}:K]^2,2\}$;
\item \label{Q2} if $K=\Q(i)$, we have $\mathfrak{f}\leq \max\{[K_{\mathfrak{f}}:K]^2,5\}$;
\item \label{Q3} if $K=\Q(\zeta_3)$, we have $\mathfrak{f}\leq \max\{[K_{\mathfrak{f}}:K]^2,7\}$;
\item \label{all} in all other cases, $\mathfrak{f}\leq [K_{\mathfrak{f}}:K]^2$.
\end{enumerate}
In all cases,
  \[\mathfrak{f}\leq 3\cdot [K_{\mathfrak{f}}:K]^2.\]
\end{proposition}

\begin{rmk} We note that the bounds given in Proposition~\ref{prop:aux_cond} are far from optimal, as is clear by considering, for example, the case when $\frakf = 1$ (i.e. when $\calO_\frakf = \calO_K$).
\end{rmk}

\begin{proof}
Recall the well-known formula for the class number (see e.g. \cite[Theorem 7.24]{Cox})
\begin{equation} \label{eq:form} [K_\mathfrak{f}:K] = h(\calO_\mathfrak{f}) =  \frac{h_K \cdot \mathfrak{f}}{[\calO_K^\times:\calO_\mathfrak{f}^\times]}\cdot \prod_{p\mid \mathfrak{f}} \left( 1 - \left( \frac{\Delta_K}{p} \right)\frac{1}{p}\right),
\end{equation}
where the symbol  $\left( \frac{\Delta_K}{p} \right)$ denotes
the Legendre symbol for odd primes, while for the prime $2$, the Legendre symbol is
replaced by the Kronecker symbol $\left( \frac{\Delta_K}{2} \right)$, defined as
\[\left( \frac{\Delta_K}{2} \right):=\begin{cases}0 & \textrm{if } 2\mid\Delta_K\\
1 & \textrm{if } \Delta_K\equiv  1\pmod{8}\\
-1 & \textrm{if } \Delta_K\equiv  5\pmod{8}.
\end{cases}\]
Then
\begin{equation}\label{eq:cond1}
\mathfrak{f}=h_K^{-1}\cdot[K_\mathfrak{f}:K]\cdot [\calO_K^\times:\calO_\mathfrak{f}^\times]\cdot \prod_{p\mid \mathfrak{f}} \frac{p}{ p - \left( \frac{\Delta_K}{p} \right)}.
\end{equation}
On the other hand, since $\prod_{p\mid \mathfrak{f}}p \leq \mathfrak{f}$  we obtain
\[\prod_{p\mid \mathfrak{f}}p \leq h_K^{-1}\cdot [K_\mathfrak{f}:K]\cdot [\calO_K^\times:\calO_\mathfrak{f}^\times]\cdot\prod_{p\mid \mathfrak{f}} \frac{p}{ p - \left( \frac{\Delta_K}{p} \right)}\]
and hence
\begin{equation}\label{eq:cond2}
\prod_{p\mid \mathfrak{f}} \left( p - \left( \frac{\Delta_K}{p} \right)\right) \leq h_K^{-1}\cdot[K_\mathfrak{f}:K]\cdot  [\calO_K^\times:\calO_\mathfrak{f}^\times].
\end{equation}

Now we prove statements \eqref{Q7} -- \eqref{all}. If we can show that
\begin{equation}\label{eq:aim}
  [\calO_K^\times:\calO_\mathfrak{f}^\times]^{2} \cdot \prod_{p\mid \mathfrak{f}} p \leq h_K^2\cdot \prod_{p\mid \mathfrak{f}} \left( p - \left( \frac{\Delta_K}{p} \right)\right)^2
  \end{equation}
then rearranging gives 
\[\prod_{p\mid \mathfrak{f}} \frac{p}{ p - \left( \frac{\Delta_K}{p} \right)}\leq h_K^2\cdot [\calO_K^\times:\calO_\mathfrak{f}^\times]^{-2}\cdot \prod_{p\mid \mathfrak{f}} \left( p - \left( \frac{\Delta_K}{p} \right)\right)\]
and substituting this into~\eqref{eq:cond1} and applying~\eqref{eq:cond2} yields $\mathfrak{f}\leq [K_\mathfrak{f}:K]^2$. We will show that~\eqref{eq:aim} holds except in some exceptional cases as described in statements~\eqref{Q7} -- \eqref{Q3}.

Since $K$ is an imaginary quadratic field, $\#\calO_K^\times \leq 6$. For any $\mathfrak{f}$, we have $\pm 1\in \calO_\mathfrak{f}^\times$, whereby $[\calO_K^\times:\calO_\mathfrak{f}^\times]\leq 3$.
First,we  consider the case where $[\calO_K^\times:\calO_\mathfrak{f}^\times]=1$.  For all primes $p\geq 3$, we have \[p\leq (p-1)^2\leq \left(p - \left( \frac{\Delta_K}{p}\right)\right)^2.\] 
Moreover, if $h_K>1$, then $2< h_K^2\leq h_K^2\cdot\left( 2 - \left( \frac{\Delta_K}{2} \right)\right)^2$. We only run into trouble in proving~\eqref{eq:aim} if $h_K=1$ and $\left( \frac{\Delta_K}{2}\right)=1$. The unique imaginary quadratic field satisfying these two hypotheses is $K=\Q(\sqrt{-7})$ with $\Delta_K=-7$. In this case, we have $2\cdot p\leq \left( p - \left( \frac{-7}{p} \right)\right)^2$ for all primes $p\geq 3$ so the only way the product $\prod_{p\mid \mathfrak{f}} \left( p - \left( \frac{\Delta_K}{p} \right)\right)^2$ can be less than $\prod_{p\mid \mathfrak{f}} p$ is if $\mathfrak{f}=2^n$ for some $n\geq 1$. In this case,~\eqref{eq:form} gives
\[[K_\mathfrak{f}:K]=2^n\cdot\left(1-\frac{1}{2}\right)=2^{n-1}.\]
Hence, $\mathfrak{f}\leq  [K_\mathfrak{f}:K]^2$ unless $\mathfrak{f}=2$.

Now consider the special case where $[\calO_K^\times:\calO_\mathfrak{f}^\times]=2$, meaning that $K=\Q(i)$, $\Delta_K=-4$ and $\mathfrak{f}>1$. For $p=3$ and all primes $p\geq 7$ we have $4\cdot p\leq  \left( p - \left( \frac{-4}{p} \right)\right)^2$. For $p\in\{2,5\}$ we have $p\leq\left( p - \left( \frac{-4}{p} \right)\right)^2$. So we only run into trouble in proving~\eqref{eq:aim} if $\mathfrak{f}=2^a\cdot 5^b$ for some $a,b\in\Z_{\geq 0}$. In this case,~\eqref{eq:form} gives
\[[K_\mathfrak{f}:K]  = 
  2^{a-1}\cdot 5^b \cdot \prod_{p\mid \mathfrak{f}} \left( 1 - \left( \frac{-4}{p} \right)\frac{1}{p}\right)\\
 = \begin{cases}
2^{a-1} & \textrm{if } b=0\\
2^{a+1}\cdot 5^{b-1}& \textrm{if } b\geq 1.
\end{cases}
\]
Now observe that $[K_\mathfrak{f}:K]^2\geq \mathfrak{f}$ unless $\mathfrak{f}\in\{2,5\}$.

Finally, consider the special case where $[\calO_K^\times:\calO_\mathfrak{f}^\times]=3$, meaning that $K=\Q(\zeta_3)$, $\Delta_K=-3$ and $\mathfrak{f}>1$. For $p\geq 11$ we have
$9 \cdot p\leq (p-1)^2\leq \left( p - \left( \frac{-3}{p} \right)\right)^2$ and for all $p$ we have $3\cdot p\leq \left( p - \left( \frac{-3}{p} \right)\right)^2$. Therefore, the only way that the product $\prod_{p\mid \mathfrak{f}} \left( p - \left( \frac{-3}{p} \right)\right)^2$ can be less than $9\cdot \prod_{p\mid \mathfrak{f}} p$ is if $\mathfrak{f}\in\{2,3,5,7\}$.

To see that in all cases $\mathfrak{f}\leq 3\cdot [K_{\mathfrak{f}}:K]^2$, observe that in the exceptional cases with $\frakf >[K_\frakf:K]^2$, we have $\frakf\leq 7$ and if $\frakf >3$ then $[K_{\mathfrak{f}}:K]=2$.
\end{proof}

\begin{remark}
An alternative bound is given by \[\mathfrak{f}\leq \max\{[K_\mathfrak{f}:K]^2, \frac{5}{2}\cdot [\calO_K^\times:\calO_\mathfrak{f}^\times]\cdot[K_\mathfrak{f}:K]\}.\] It can be easily checked that the inequality $\mathfrak{f}\leq\frac{5}{2}\cdot[\calO_K^\times:\calO_\mathfrak{f}^\times]\cdot[K_\mathfrak{f}:K] $ holds in all the exceptional cases in the proof of Proposition~\ref{prop:aux_cond}.
\end{remark}

\begin{cor}\label{cor:conductor}
Let $k$ be a number field and let $E/k$ be an elliptic curve with CM by an order of conductor $\mathfrak{f}$ in an imaginary quadratic field $K$. Then $K_{\mathfrak{f}}\subset Kk$ and hence $\frakf$ satisfies the bounds of Proposition~\ref{prop:aux_cond} with $[Kk:K]$ or $[k:\Q]$ in place of $[K_{\mathfrak{f}}:K]$. In particular,
\[\frakf\leq \max\{[k:\Q]^2,7\} \textrm{ and}\]
\[\frakf\leq 3\cdot[k:\Q]^2.\]
\end{cor}

\begin{proof}
The theory of complex multiplication tells us that $K_\mathfrak{f}=K(j(E))$. Since $E$ is defined over $k$, we have $K(j(E))\subset Kk$.
\end{proof}

\begin{cor}\label{cor:EC_CM_K}
Let $d\in\Z_{>0}$ and let $K$ be an imaginary quadratic field. 
Then the number of $\CC$-isomorphism classes of elliptic curves defined over number fields of degree at most $d$ with (not necessarily full) CM by $K$ is equal to
\begin{align*}
2 & \textrm{ if } d=1 \textrm{ and } K\in\{\Q(\sqrt{-7}),\Q(i)\};\\
3 & \textrm{ if } d=1 \textrm{ and } K=\Q(\zeta_3);\\
9 & \textrm{ if } d=2 \textrm{ and } K=\Q(\zeta_3).
\end{align*}
In all other cases, the number of $\CC$-isomorphism classes of elliptic curves defined over number fields of degree at most $d$ with (not necessarily full) CM by $K$ is bounded above by $d^3$. 
\end{cor}

\begin{proof}
Let $\calO_\mathfrak{f}$ denote the order of conductor $\mathfrak{f}$ in $\calO_K$. Then the theory of complex multiplication shows that the number of isomorphism classes of complex elliptic curves with CM by $\calO_\mathfrak{f}$ is equal to the class number $h(\calO_\mathfrak{f})$. We call a conductor $\mathfrak{f}$ $d$-\emph{permissible} if there exists an
elliptic curve $E$ defined over a number field of degree at most $d$ with $\End (\overline{E})=\calO_\mathfrak{f}$. In this case the theory of complex multiplication shows that $h(\calO_\mathfrak{f})=[K(j(E)):K]\leq d$.
The total number of $\CC$-isomorphism classes of elliptic curves defined over number fields of degree at most $d$ with CM by $K$ is given by 
\[\sum_{d-\textup{permissible }\mathfrak{f}}h(\calO_\mathfrak{f})\leq \sum_{d-\textup{permissible }\mathfrak{f}}d.\]
Corollary~\ref{cor:conductor} and Proposition~\ref{prop:aux_cond} show that in most cases if $\mathfrak{f}$ is $d$-permissible then $\mathfrak{f}\leq d^2$, which gives the desired result. The exceptional cases are
\begin{enumerate}
\item  $K=\Q(\sqrt{-7})$ and $d=1$, in which case $\mathfrak{f}\leq 2$;
\item $K=\Q(i)$ and $d\leq 2$, in which case $\mathfrak{f}\leq 5$;
\item $K=\Q(\zeta_3)$ and $d\leq 2$, in which case $\mathfrak{f}\leq 7$.
\end{enumerate}
The results for the exceptional cases listed above with $d=1$ are well known -- see \cite[Appendix~A \S 3]{SilvermanII}, for example. It remains to tackle cases (2) and (3) when $d=2$. For this, we use Corollary~\ref{cor:conductor} and Proposition~\ref{prop:aux_cond}.

First we tackle case (2). The number of $\CC$-isomorphism classes of elliptic curves defined over number fields of degree at most $2$ with CM by $\Q(i)$ is given by 
\begin{equation}\label{eq:hi}
\sum_{\substack{\mathfrak{f}\leq 5\\ 2-\textup{permissible }\mathfrak{f}}}h(\calO_\mathfrak{f}).
\end{equation}
We calculate that $h(\Z[i])=h(\Z[2i])=1$ and $h(\calO_\mathfrak{f})=2$ for $3\leq \mathfrak{f}\leq 5$.  So $\sum_{\mathfrak{f}=1}^5h(\calO_\mathfrak{f})=1+1+2+2+2=8=d^3$ and so this case is compatible with the usual bound for the non-exceptional cases.

Now we tackle case (3). The number of $\CC$-isomorphism classes of elliptic curves defined over number fields of degree at most $2$ with CM by $\Q(\zeta_3)$ is given by 
\begin{equation}\label{eq:hz}
\sum_{\substack{\mathfrak{f}\leq 7\\ 2-\textup{permissible }\mathfrak{f}}}h(\calO_\mathfrak{f}).
\end{equation}
We calculate that $h(\calO_\mathfrak{f})=1$ for $1\leq \mathfrak{f}\leq 3$, $h(\calO_\mathfrak{f})=2$ for $\mathfrak{f}\in\{4,5,7\}$ and $h(\calO_6)=3$, so $6$ is not $2$-permissible.  Using Sage \cite{Sage} for example, one can check that the other values of $\mathfrak{f}$ are all $2$-permissible -- there are two non-rational CM $j$-invariants defined over $\Q(\sqrt{3})$ with CM by the order of conductor $4$ in $\Z[\zeta_3]$, for example.
\end{proof}

\begin{proposition}\label{prop:ECMbound}
Let $d\in\Z$ with $d\geq 2$.
Then the number of $\CC$-isomorphism classes of elliptic curves defined over number fields of degree at most $d$ with (not necessarily full) CM is bounded above by 
\[d^3\cdot\# \{K \textup{ imaginary quadratic}\mid h_K\leq d\}.\]
The number of $\CC$-isomorphism classes of elliptic curves defined over $\Q$ with (not necessarily full) CM is $13$.
\end{proposition}

\begin{proof}
We say that an imaginary quadratic field $K$ is $d$-permissible if there exists an
elliptic curve $E$ defined over a number field of degree at most $d$ with $\End (\overline{E})=\calO$ for some order $\calO$ in $K$. In this case, since $h_K\mid h(\calO)$ and $h(\calO)=[K(j(E)):K]\leq d$, we have $h_K\leq d$. The result will follow from applying Corollary~\ref{cor:EC_CM_K} and summing over $d$-permissible fields $K$.
The result for $\Q$ is well known. It follows from Corollary~\ref{cor:EC_CM_K} and the fact that there are $9$ imaginary quadratic fields $K$ with class number $1$. For $d\geq 3$, the result follows immediately from Corollary~\ref{cor:EC_CM_K}. Now suppose that $d=2$. Corollary~\ref{cor:EC_CM_K} shows that the contribution from $\Q(\zeta_3)$ is $9$, rather than $d^3=8$. However, this is compensated for by the fact that the contribution from $\Q(\sqrt{-7})$ is at most $4$, as we now show. 
The contribution from $\Q(\sqrt{-7})$ is given by 
\[\sum_{2-\textup{permissible }\mathfrak{f}}h(\calO_\mathfrak{f})\]
and if $\mathfrak{f}$ is $2$-permissible then $h(\calO_\mathfrak{f})$ is at most $2$, whence Proposition~\ref{prop:aux_cond} shows that $\mathfrak{f}\leq 4$. Now \eqref{eq:form} gives
\[h(\calO_\mathfrak{f})=\mathfrak{f}\cdot \prod_{p\mid \mathfrak{f}} \left( 1 - \left( \frac{-7}{p} \right)\frac{1}{p}\right)\]
whereby $h(\calO_3)=4$ and hence $3$ is not $2$-permissible. Thus, the contribution from $\Q(\sqrt{-7})$ is bounded above by $h(\calO_{\Q(\sqrt{-7})})+h(\calO_2)+h(\calO_4)=1+1+2=4$.
\end{proof}

\begin{remark}
In \cite[Theorem~1.1]{DanielsLozanoRobledo}, Daniels and Lozano-Robledo show that if $k/\Q$ is an extension of odd degree then the number of distinct CM $j$-invariants defined over $k$ is at most $13+2\log _3([k:\Q])$. However, the odd degree case is very rare -- in \cite[Corollary~2.4]{DanielsLozanoRobledo} the authors show that if $K/\Q$ is an imaginary quadratic field with odd class number then $K=\Q(\sqrt{-d})$ where $d$ is equal to $1,2$ or a prime $q\equiv 3\pmod{4}$. For a numerical illustration of the scarcity of CM $j$-invariants defined over fields of odd degree compared to those defined over fields of even degree, see \cite[Table 2]{DanielsLozanoRobledo}.
\end{remark}

\begin{theorem} \label{strongSha1}
For a number field $k$, let $S_k$ denote the set of $\CC$-isomorphism classes of singular K3 surfaces $X$ such that both $X$ and a set of generators for $\NS\overline{X}$ are defined over $k$. For $d\in\Z_{>0}$, let $S_d:=\bigcup_{[k:\Q]\leq d}S_k$. Then 
\[\# S_d\leq3\cdot d^3\cdot(\log(3\cdot d^2)+1)\cdot\# \{K \textup{ imaginary quadratic}\mid h_K\leq d\}.\]
\end{theorem}

\begin{proof}
Let $k$ be a number field of degree at most $d$ and let $X\in S_k$. The isomorphism class of $X_\CC$ is determined by the isomorphism class of its transcendental lattice $T({X_{\CC}})$ (see \cite[\S14, Corollary 3.21]{Huybrechts}). Let $\mathfrak{c}:=\disc T({X_{\CC}})$ and let $K:=\Q(\sqrt{\mathfrak{c}})$. Define $\mathfrak{f}\in\Z_{>0}$ by letting $\mathfrak{c}=\mathfrak{f}^2\cdot\Delta_K$. Then by \cite[Theorem~2]{Schuett},
 the ring class field $K_\mathfrak{f}$ is contained in $kK$ and hence $[K_\mathfrak{f}:K]\leq [kK:K]\leq [k:\Q]\leq d$. Now Proposition~\ref{prop:aux_cond} shows that 
\begin{equation}\label{eq:c1}
\mathfrak{f}\leq 3[K_\mathfrak{f}:K]^2\leq 3[k:\Q]^2\leq 3\cdot d^2.
\end{equation}
Also, since the Hilbert class field is contained in $K_\mathfrak{f}$, we have $h_K\leq d$.

Work of Shioda and Inose in \cite{SI} shows that $T({X_\CC})=T({A})$ for $A=\CC/\calO_{K,\mathfrak{f}}\times \CC/\mathfrak{a}$ where $\calO_{K,\mathfrak{f}}$ denotes the order of conductor $\mathfrak{f}$ in $\calO_K$ and $\mathfrak{a}$ is a lattice in $K$ with ring of multipliers $\calO_{K,\mathfrak{f}_\mathfrak{a}}$ with $\mathfrak{f}_\mathfrak{a}\mid \mathfrak{f}$. The number of homothety classes of lattices with ring of multipliers $\calO_{K,\mathfrak{f}_\mathfrak{a}}$ is equal to the class number $h(\calO_{K,\mathfrak{f}_\mathfrak{a}})$.
Our observations thus far show that 
  \begin{equation}\label{eq:K3sum}
\# S_d\leq\sum_{\substack{K \textup{ imaginary quadratic}\\ h_K\leq d}}\ \sum_{\mathfrak{f}\leq 3\cdot d^2}\ \sum_{\mathfrak{f}_\mathfrak{a}\mid\mathfrak{f}}h(\calO_{K,\mathfrak{f}_\mathfrak{a}}).
\end{equation}
Since $\mathfrak{f}_\mathfrak{a}\mid \mathfrak{f}$, we have $K_{\mathfrak{f}_\mathfrak{a}}\subset K_\mathfrak{f}\subset kK$. Therefore, 
 \begin{equation}\label{eq:hd}
 h(\calO_{K,\mathfrak{f}_\mathfrak{a}})=[K_{\mathfrak{f}_\mathfrak{a}}:K]\leq [k:\Q]\leq d.
 \end{equation}
Substituting \eqref{eq:hd} into \eqref{eq:K3sum} and writing $\tau$ for the number-of-divisors function 
  gives 
 \begin{equation}\label{eq:K3sum2}
\# S_d\leq d\cdot  \sum_{\substack{K \textup{ imaginary quadratic}\\ h_K\leq d}}\ \sum_{\mathfrak{f}\leq 3\cdot d^2}\tau(\mathfrak{f}).
\end{equation}
Now recall that $\sum_{n=1}^M\tau(n)=\sum_{r=1}^M\lfloor{\frac{M}{r}}\rfloor\leq M\sum_{r=1}^M\frac{1}{r}\leq M(\log M+1)$. Using this in \eqref{eq:K3sum2} yields
 \begin{equation*}\label{eq:K3sum3}
\# S_d\leq 3\cdot d^3\cdot(\log(3\cdot d^2)+1)\cdot\# \{K \textup{ imaginary quadratic}\mid h_K\leq d\}.\qedhere
\end{equation*}
\end{proof}

\begin{cor} \label{strongSha}
The number of $\CC$-isomorphism classes of singular K3 surfaces defined over number fields of degree at most $d$ is bounded above by 
\[\begin{array}{l}
3\cdot M(20)^3\cdot d^3\cdot(\log(3\cdot M(20)^2\cdot d^2)+1)\\ 
 \cdot\# \{K \textup{ imaginary quadratic}\mid h_K\leq M(20)\cdot d\}.\end{array}\]
\end{cor}

\begin{proof}
Let $X$ be a singular K3 surface defined over a number field $k$. Recall that $\NS \overline{X}=\Pic\overline{X}$ is a free $\Z$-module of rank $20$ whose generators are all defined over some finite extension of $k$. Since the order of any finite subgroup of $\GL_{20}(\Z)$ divides $M(20)$,
the Galois representation $\rho:\Gamma_k\to\Aut \NS \overline{X}\hookrightarrow \GL_{20}(\Z)$ factors through $\Gal(k_0/k)$ for a Galois extension $k_0/k$ of degree at most $M(20)$. Now apply Theorem~\ref{strongSha1} to $X_{k_0}$.
\end{proof}


\section{The transcendental part of the Brauer group of the self-product of a CM elliptic curve}\label{sec:ExE}

In this section, we obtain uniform bounds for the transcendental part of the Brauer group of $E\times E$, where $E$ is an elliptic curve over a number field.
The key result that we will use to compute the transcendental part of the Brauer group of a product of elliptic curves is the following:

\begin{theorem}[Skorobogatov -- Zarhin, {\cite[Proposition 3.3]{SZ12}}]\label{thm:SZ}
Let $E$ and $E'$ be elliptic curves over a field $k$ of characteristic 0 and let $n\in\Z_{>0}$. Then there is a canonical isomorphism of abelian groups
\[\frac{\Br (E \times E')[n]}{\Br_1(E \times E')[n]}\cong \frac{\Hom_{\Gamma_k}(E[n],E'[n])}{(\Hom(\overline{E},\overline{E'})\otimes \Z/n\Z)^{\Gamma_k}}.\]
\end{theorem}

We will apply this result in the case where $E=E'$. The special case where $E$ has full CM was addressed in \cite{Newton15}. The following definition is needed for the description of the $\ell$-primary part of $\Br(E\times E)/\Br_1(E\times E)$ in Theorem~\ref{thm:rn} below.

\begin{defn}[{\cite[Definition 1]{Newton15}}]\label{Grossencharakter}
Let $E$ be an elliptic curve over a number field $k$ with CM by the ring of integers of an imaginary quadratic field $K$.
For a prime number $\ell\in\Z_{>0}$, define $m_\ell(E)$ to be the largest integer $n$ such that for all primes $\frakq$ of $kK$ that are of good reduction for $E$ and coprime to $\ell$, the Gr\"{o}ssencharakter $\psi_{E/kK}$ satisfies $$\psi_{E/kK}(\frakq) \in \mathcal{O}_{K,\ell^n},$$ 
where $ \mathcal{O}_{K,\ell^n}$ denotes the order in $\mathcal{O}_K$ of conductor $\ell^n$.
\end{defn}

\begin{theorem}[Newton]\label{thm:rn}
Let $E$ be an elliptic curve over a number field $k$ with CM by the ring of integers of an imaginary quadratic field $K$, let $\ell$ be a prime number and let $m:=m_\ell(E)$ be as defined in Definition~\ref{Grossencharakter}. Then
\begin{equation}\label{rn}
   \frac{\Br( E\times E)}{\Br_1(E\times E)}\{\ell\}\cong \begin{cases}
\big(\Z/\ell^m\Z \big)^2 & \textrm{  if  } K\subset k,\\
\Z/2^m\Z \times \Z/2\Z & \textrm{  if  } K\not\subset k, \ell= 2 \text{  and  } E[2]= E[2](k),\\
\Z/\ell^m\Z  &\textrm{  otherwise.}\\
\end{cases}
\end{equation}
\end{theorem}

\begin{proof}
See \cite[Theorems~2.5 and 2.9]{Newton15}.
\end{proof}

\begin{remark}\label{rmk:m>0}
Since the Gr\"{o}ssencharakter determines the action of $\Gamma_{kK}$ on $E[2]$, we note that $E[2]= E[2](k)$ implies $m_2(E)\geq 1$. If $K\not\subset k$, then $E[2]= E[2](k)$ also implies $2\mid\Delta_K$. This is seen by taking a basis $\Bigl\{P, \Bigl(\frac{\Delta_K+\sqrt{\Delta_K}}{2}\Bigr)P\Bigr\}$ for $E[2]$ and considering the action of complex conjugation.
\end{remark}

In order to use Theorem~\ref{thm:rn} to obtain uniform bounds on the size of the transcendental part of the Brauer group of $E\times E$, we need to bound $\prod_{\ell \textup{ prime}} \ell^{ m_\ell(E)}$ in terms of the degree of the field of definition of $E$. This is achieved by the following lemma in combination with Proposition~\ref{prop:aux_cond}.

\begin{lemma}\label{lem:Kf}
Let $k$ be a number field. Let $E$ be a CM elliptic curve over $k$ with $\End(\overline{E})=\calO_K$ for some imaginary quadratic field $K/\Q$. Let $m_\ell :=m_\ell(E)$ be as in Definition \ref{Grossencharakter}. Let $\mathfrak{c}:= \prod_{\ell \textup{ prime}} \ell^{ m_\ell}$ and let $K_\mathfrak{c}$ denote the ring class field associated to $\calO_\mathfrak{c}$. Then
\[K_\mathfrak{c}\subset kK.\]
\end{lemma}

\begin{proof}
Let $\frakq$ be a finite prime of $kK$ of good reduction for $E$ and coprime to $\mathfrak{c}$. We will show that $\frakq$ splits completely in $kK_\mathfrak{c}/ kK$. Then \cite{CasselsFrohlich}[Exercise~6.1] will allow us to conclude that $K_\mathfrak{c}\subset kK$, as desired. 

Recall that, given an abelian extension of number fields $M/F$ and a prime ideal $\frakr$
of $\calO_F$ that is unramified in $M/F$, the Artin symbol $(\frakr, M/F)$ is the unique element
$\sigma\in \Gal(M/F)$ such that, for all $\alpha \in \calO_M$,
\[\sigma(\alpha)\equiv \alpha^{N_{F/\Q}(\frakr)}\pmod{\fraks}\]
where $\fraks$ is a prime of $M$ above $\frakr$. Showing that $\frakq$ splits completely in $kK_\mathfrak{c}/ kK$ is equivalent to showing that $(\frakq, kK_{\frakc}/kK)=1$. It will suffice to show that the restriction $(\frakq, kK_{\frakc}/kK)|_{K_\frakc}\in \Gal(K_\frakc/K)$ is trivial. Let $N_{kK/K}(\frakq) = \frakp^{f_{\frakq/\frakp}}$, where $\frakp:=\frakq\cap\calO_K$ and $f_{\frakq/\frakp}:=[\calO_{kK}/\frakq:\calO_K/\frakp]$. Then $N_{kK/\Q}(\frakq)=N_{K/\Q}(\frakp)^{f_{\frakq/\frakp}}$. We have
\begin{equation}\label{eq:res}
(\frakq, kK_{\frakc}/kK)|_{K_\frakc}=(\frakp,K_\frakc/K)^{f_{\frakq/\frakp}}=(\frakp^{f_{\frakq/\frakp}},K_\frakc/K).
\end{equation}
 By \cite[Theorems~II.9.1 and
II.9.2]{SilvermanII}, the value $\psi_{E/kK}(\frakq)$ of the Gr\"{o}ssencharakter at $\frakq$ generates the principal ideal $N_{kK/K}(\frakq) = \frakp^{f_{\frakq/\frakp}}$. By definition of $m_\ell$, we have $\psi_{E/kK}(\frakq)\in\calO_{\ell^{m_\ell}}=\Z+\ell^{m_\ell}\calO_K$ for all prime numbers $\ell$. Thus, $\psi_{E/kK}(\frakq)\in \bigcap_\ell \calO_{\ell^{m_\ell}}=\calO_\frakc=\Z+\frakc\calO_{K}$. By definition of the ring class field $K_\frakc$, this implies that $(\frakp^{f_{\frakq/\frakp}},K_\frakc/K)=((\psi_{E/kK}(\frakq)),K_\frakc/K)=1$, as required.
\end{proof}

Now we deal with the more general case where the elliptic curve $E$ has CM by an order in the ring of integers of an imaginary quadratic field. In the next lemma we compute $(\End (\overline{E}) \otimes \Z/n\Z)^{\Gamma_k}$.

\begin{lemma}\label{lem:denom}
Let $E$ be an elliptic curve over a field $k$ of characteristic 0, with CM by an order of conductor $\mathfrak{f}$ in an imaginary quadratic field $K$. Let $n\in\Z_{>0}$. 
\begin{enumerate}
\item If $K\subset k$ then $(\End( \overline{E})\otimes \Z/n\Z)^{\Gamma_k}=\End(\overline{E})\otimes \Z/n\Z\cong(\Z/n\Z)^2.$
\item If $K\not\subset k$ then 
\[(\End(\overline{E})\otimes \Z/n\Z)^{\Gamma_k}\cong\begin{cases}
\Z/n\Z\times \Z/2 \Z& \textrm{if } 2\mid\mathfrak{f}\cdot\Delta_K \textrm{ and } 2\mid n;\\
\Z/n\Z & \textrm{otherwise}.
\end{cases}\]
\end{enumerate}
\end{lemma}

\begin{proof}
If $K\subset k$, then $\Gamma_k$ acts trivially on $\End(\overline{E})$, and (1) follows immediately. It remains to prove (2). Henceforth, we assume that $K=\Q(\sqrt{-d})\not\subset k$. 

First suppose that $2\nmid\Delta_K$. Then any $\phi\in\End(\overline{E})$ is of the form $a+b\frakf(\frac{1+\sqrt{-d}}{2})$ for some $a,b,\in \Z$, and a simple calculation shows that the image of $\phi$ in $\End(\overline{E})\otimes \Z/n\Z$ is fixed by $\Gamma_k$ if and only if $2b\equiv\mathfrak{f}b\equiv 0\pmod{n}$. If either $\frakf$ or $n$ is odd then these congruences imply that $b\equiv 0\pmod n$ and hence $(\End(\overline{E})\otimes \Z/n\Z)^{\Gamma_k}\cong\Z/n\Z$. If both $\frakf$ and $n$ are even then the two congruences simply reduce to $2b\equiv 0\pmod{n}$ and $(\End( \overline{E})\otimes \Z/n\Z)^{\Gamma_k}\cong \Z/n\Z\times \Z/2\Z$, as claimed.

Now suppose that $2\mid\Delta_K$. Then any $\phi\in\End(\overline{E})$ is of the form $a+b\frakf\sqrt{-d}$ for some $a,b,\frakf\in \Z$, and the image of $\phi$ in $\End(\overline{E})\otimes \Z/n\Z$ is fixed by $\Gamma_k$ if and only if $2b\equiv 0\pmod{n}$. This yields the desired result.
\end{proof}

To make use of Theorem~\ref{thm:SZ}, we must also analyse $\End_{\Gamma_k} (E[n])$. For this we use some ideas from~\cite{VAV16} and~\cite{VAVCorrigendum}. Let $n$ be a positive integer and let $E$
be an elliptic curve over a field $k$ of characteristic $0$. Let
\[\rho_{E,n} : \Gamma_k \to \Aut( E[n] )\cong \GL_2(\Z/n\Z)\]
denote the Galois representation coming from the action of Galois on the $n$-torsion of $E$.  

\begin{defn}[{\cite[Definition~A.1]{VAVCorrigendum}}]
Let $n$ be a positive integer. A subgroup $H$ of $\GL_2(\Z/n\Z)$ is liftable abelian if
there exists an abelian subgroup $\widehat{H} < \GL_2(\widehat{\Z})$ such that $\widehat{H}$ surjects onto $H$ under the natural
quotient map $\GL_2(\widehat{\Z})\to\GL_2(\Z/n\Z)$. (In particular, a liftable abelian subgroup is abelian.)
\end{defn}

\begin{proposition}[{\cite[Corollary~A.4]{VAVCorrigendum}}]\label{prop:VAVcorrected}
Let $E$ be an elliptic curve over a field $k$ of characteristic $0$ and let $n\in\Z_{>0}$. Then we
have an isomorphism of abelian groups
\[\End_{\Gamma_k} (E[n]) \cong \Z/n\Z \times \Z/s_{E,n}\Z \times (\Z/t_{E,n}\Z)^2\]
for positive integers $t_{E,n}\mid s_{E,n}\mid n$. Furthermore, $s_{E,n}$ is the largest integer $s$ dividing $n$ such that $\Gal(k(E[s])/k)$ is liftable abelian and $t_{E,n}$ is the largest integer $t$ dividing $n$ such that $\Gal(k(E[t])/k) \subset (\Z/t\Z)^\times$ where $a \in (\Z/t\Z)^\times$ acts by $P\mapsto aP$.
\end{proposition}

\begin{remark}
 An example where $\im\rho_{E,n}$ is abelian but not liftable abelian is as follows. Take $k=\Q(\sqrt{2})$ and let $E/k$ be the elliptic curve 64.1-a3 in the LMFDB tables (see \cite[\href{www.lmfdb.org/EllipticCurve/2.2.8.1/64.1/a/3)}{Elliptic Curve 64.1-a3}]{lmfdb}) with CM by $\Z[\sqrt{-2}]$. Choose a basis of the form $P, \sqrt{-2}P$ for $E[4]$. With respect to such a basis, one can calculate using the methods of \cite{Newton15} that the $\Z/4\Z$-submodule of $\End( E[4])$ generated by $\im\rho_{E,4}$ is equal to the $\Z/4\Z$-span of $I$, $\begin{pmatrix}2&0\\0&0\end{pmatrix}$ and $\begin{pmatrix}0&0\\2&0\end{pmatrix}$. Thus, $\im\rho_{E,4}$ is abelian but \cite[Lemma~A.7]{VAVCorrigendum} shows that it is not liftable abelian.
\end{remark}

\begin{lemma}\label{lem:alexis}
Let $E$ and $E'$ be elliptic curves over a field $k$ of characteristic $0$ and let $\varphi: E\to E'$ be an isogeny of degree $d$ defined over $k$. Then for all primes $\ell$ and all $n\in \Z_{\geq 0}$, $\varphi$ induces a Galois-equivariant surjection
\[\varphi: E[\ell^{n+\ord_\ell d}]\twoheadrightarrow E'[\ell^n]\]
and hence a surjection
\[\Gal(k(E[\ell^{n+\ord_\ell d}])/k)\twoheadrightarrow \Gal(k(E'[\ell^n])/k).\]
\end{lemma}

\begin{proof}
Let $P' \in E'[\ell^n]$.  Then $P'=\varphi(P)$ for some $P\in E(\kbar)$.  Since \[[\ell^n]P'=([\ell^n]\circ \varphi)(P)=(\varphi\circ [\ell^n])(P),\] we have $[\ell^n]P\in\ker(\varphi)$.  Writing $\hat{\varphi}$ for the dual isogeny, we have \[[d\ell^n]P=(\hat{\varphi}\circ \varphi)([\ell^n]P)=0.\]  Therefore, $P=P_1+P_2$ for some $P_1\in E[\ell^{n+\ord_\ell d}]$ and $P_2\in E[\frac{d}{\ell^{\ord_\ell d}}]$. Since $\varphi(P_1)$ is a point of $E'$ with order a power of $\ell$, the same is true for $\varphi(P_2)=P'-\varphi(P_1)$.  Since $P_2\in E[\frac{d}{\ell^{\ord_\ell d}}]$, and $\ell\nmid\frac{d}{\ell^{\ord_\ell d}}$, we deduce that $\varphi(P_2)=0$, and hence $P'=\varphi(P_1)$. This proves the existence of the first surjection. Since $\varphi$ is defined over $k$, it is Galois equivariant, whence the second surjection.
\end{proof}

The following fact is well known, but we give a proof here since we were unable to find one in the literature.

\begin{lemma}\label{isogenoustomaximal}Let $k$ be a field of characteristic $0$.
Let $E$ be an elliptic curve over $k$ with CM by an order $\calO$ of conductor $\mathfrak{f}$ in an imaginary quadratic field $K$. Then there exists a cyclic $k$-isogeny $\varphi: E\to E'$ of degree $\mathfrak{f}$, where $E'$ is an elliptic curve over $k$ with CM by $\calO_K$. 
\end{lemma}

\begin{proof}
The complex elliptic curve $E_\CC$ corresponds to $\CC/L$ for some lattice $L$. Since $E$ has CM by $\calO$, the lattice $L$ is homothetic to $\mathfrak{b}$ for some invertible $\mathcal{O}$-ideal $\mathfrak{b}$. In other words $L=\lambda\mathfrak{b}$ for some $\lambda\in\CC^\times$. Now the natural surjection $\CC/\lambda\mathfrak{b}\twoheadrightarrow \CC/\lambda\mathfrak{b}\mathcal{O}_K$ corresponds to a cyclic $\CC$-isogeny $\varphi: E_\CC\to E'_\CC$ where $E'_\CC$ has CM by $\calO_K$. Moreover, $\deg\varphi=[\lambda\mathfrak{b}\mathcal{O}_K:\lambda\mathfrak{b}]=[\calO_K:\calO]=\mathfrak{f}$. 

Now let $L'$ be an arbitrary lattice containing $\lambda\mathfrak{b}$ as a sublattice of index $\mathfrak{f}$ such that $\{z\in\CC\mid zL'\subset L'\}=\mathcal{O}_K$. Any such lattice corresponds to an elliptic curve $E''_\CC$ with CM by $\calO_K$ and with a $\CC$-isogeny $E_\CC \to E''_\CC$ of degree $\frakf$. Now $ \lambda^{-1}\mathfrak{b}^{-1}L'$ contains $\calO$ as a sublattice of index $\mathfrak{f}$. Furthermore, since $\{z\in\CC\mid zL'\subset L'\}=\mathcal{O}_K$, we know that $L'$ is homothetic to an invertible $\mathcal{O}_K$-ideal. Therefore, we can write $ \lambda^{-1}\mathfrak{b}^{-1}L'=\mu\mathfrak{a}$ for some invertible $\mathcal{O}_K$-ideal $\mathfrak{a}$ and some $\mu\in\CC^\times$.
Since $1\in \calO\subset \mu\mathfrak{a}$, we have $\mu=a^{-1}$ for some $a\in \mathfrak{a}$. Writing out $[a^{-1}\calO_K:\calO]$ in two different ways gives
\[[a^{-1}\calO_K:\calO_K][\calO_K:\calO]=[a^{-1}\calO_K:a^{-1}\mathfrak{a}][a^{-1}\mathfrak{a}:\calO].\]
Since $[a^{-1}\mathfrak{a}:\calO]=[ \lambda^{-1}\mathfrak{b}^{-1}L':\calO]=\mathfrak{f}=[\calO_K:\calO]$ we obtain \[[a^{-1}\calO_K:\calO_K]=[a^{-1}\calO_K:a^{-1}\mathfrak{a}].\] Therefore $ \lambda^{-1}\mathfrak{b}^{-1}L'=a^{-1}\mathfrak{a}=\calO_K$ and hence $L'=\lambda\mathfrak{b}\mathcal{O}_K$.
Thus any $\CC$-isogeny of degree $\mathfrak{f}$ to an elliptic curve with CM by $\calO_K$ has the same kernel as $\varphi$. In particular, noting that any $\CC$-isogeny is actually already defined over $\kbar$,  if $\sigma \in \Gamma_k$ then $\sigma\circ \varphi \circ \sigma^{-1}$ gives a $\CC$-isogeny of degree $\mathfrak{f}$ to an elliptic curve with CM by $\calO_K$. Its kernel is $\sigma(\ker\varphi)$. Hence, by the argument above, $\sigma(\ker\varphi) = \ker \varphi$. Since this is true for any $\sigma \in \Gamma_k$,  we deduce that $\ker\varphi$ is defined over $k$. Now the natural surjection $E\to E/\ker\varphi$ is our desired $k$-isogeny. 
\end{proof}

\begin{cor}\label{cor:stcomparison}
Let $k$ be a field of characteristic $0$.
Let $E$ be an elliptic curve over $k$ with CM by an order of conductor $\mathfrak{f}$ in an imaginary quadratic field $K$. Let $\varphi: E\to E'$ be an isogeny of degree $\mathfrak{f}$ to an elliptic curve with CM by $\calO_K$ as in Lemma~\ref{isogenoustomaximal}. Let $\ell$ be a prime and let $n\in\Z_{\geq 0}$. Then
\[s_{E,\ell^{n+\ord_\ell\frakf}}\leq \ell^{\ord_\ell\frakf}\cdot s_{E',\ell^n} \textrm{ and } t_{E,\ell^{n+\ord_\ell\frakf}}\leq \ell^{\ord_\ell\frakf}\cdot t_{E',\ell^n}.\]
\end{cor}

\begin{proof}
This follows from Lemma~\ref{lem:alexis}.
\end{proof}

We will use the following well-known fact several times in the proof of Theorem~\ref{thm:Alexis} below. Let $E, E'$ be elliptic curves over a number field. Then 
\begin{equation}\label{eq:geomBrauer}
\Br( \overline{E}\times \overline{E'}) \cong (\Q/\Z)^{6-\rho}
\end{equation} 
where $\rho=\rank\NS(\overline{E}\times \overline{E'})$.
This follows from work of Grothendieck.

\begin{theorem}\label{thm:Alexis}
Let $E$ be an elliptic curve over a number field $k$ with CM by an order of conductor $\mathfrak{f}$ in an imaginary quadratic field $K$. Let $\varphi: E\to E'$ be an isogeny of degree $\mathfrak{f}$ to an elliptic curve with CM by $\calO_K$ as in Lemma~\ref{isogenoustomaximal}. Let $\ell$ be a prime. 
\begin{enumerate}
\item \label{CMoverk} If $K\subset k$, then
\[\frac{\Br (E \times E)}{\Br_1(E \times E)}\{\ell\} \cong (\Z/\ell^{a}\Z)^2\]
for some $a\leq m_{\ell}(E')+\ord_\ell\frakf$.
\item \label{CMnotoverk} If $K\not\subset k$, then
\[\frac{\Br (E \times E)}{\Br_1(E \times E)}\{\ell\}\cong
\Z/\ell^{a}\Z \]
for some $a\leq m_{\ell}(E')+ \ord_\ell\frakf$, unless $\ell=2$ and $E[2]=E[2](k)$, in which case
\[\frac{\Br (E \times E)}{\Br_1(E \times E)}\{2\}\cong
\Z/2^{a}\Z\times\Z/2\Z \]
for some $a\leq m_{2}(E')+\ord_2\frakf+1$. In fact, $a\leq m_{2}(E')+\ord_2\frakf$ unless $2\mid \Delta_K$ 
and $E'[2]\neq E'[2](k)$. 
\end{enumerate}
\end{theorem}

\begin{remark}\label{rmk:implications}
Note that if $K\not\subset k$, then $E[2]=E[2](k)$ implies $2\mid\frakf\cdot \Delta_K$. This is seen by considering the action of complex conjugation on $E[2]$, as in Remark~\ref{rmk:m>0}. Furthermore, if $2\nmid\frakf$ then Lemma~\ref{lem:alexis} shows that $E[2]=E[2](k)$ if and only if $E'[2]=E'[2](k)$.
\end{remark}

\begin{proof}[Proof of Theorem~\ref{thm:Alexis}]
By \cite[Lemma~2.1]{Newton15}, we have
 \[\frac{\Br (E \times E)}{\Br_1(E \times E)}\{\ell\}=\frac{\Br (E \times E)\{\ell\}}{\Br_1(E \times E)\{\ell\}}.\]  Let $n\in\Z_{\geq 0}$ and apply Theorem~\ref{thm:SZ} to $E\times E$. Let $m=m_{\ell}(E')$.

\eqref{CMoverk} If $K\subset k$ then Lemma~\ref{lem:denom} and Proposition~\ref{prop:VAVcorrected} give 
\[\frac{\End_{\Gamma_k}(E[{\ell^n}])}{(\End(\overline{E})\otimes \Z/\ell^n\Z)^{\Gamma_k}}\cong( \Z/t_{E,\ell^n}\Z)^2.\]
To see this, apply Proposition~\ref{prop:VAVcorrected} to get \[\End_{\Gamma_k}(E[{\ell^n}])\cong  \Z/\ell^n\Z \times \Z/s_{E,\ell^n}\Z \times (\Z/t_{E,\ell^n}\Z)^2.\] Now, by Lemma~\ref{lem:denom}, we have \[(\End( \overline{E})\otimes \Z/\ell^n\Z)^{\Gamma_k}\cong(\Z/\ell^n\Z)^2\subset \End_{\Gamma_k}(E[{\ell^n}]).\] Since $t_{E,\ell^n}\mid s_{E,\ell^n}\mid \ell^n$, it follows that $s_{E,\ell^n}=\ell^n$, whence the claim. 

Our task is now to bound $t_{E,\ell^n}$ for $n$ large.
By Corollary~\ref{cor:stcomparison} it suffices to show that, for all $r\in\Z_{\geq 0}$, $t_{E',\ell^r}\leq \ell^m$. This follows from Lemma~\ref{lem:denom}, Proposition~\ref{prop:VAVcorrected} and Theorem~\ref{thm:rn} applied to $E'$. 

\eqref{CMnotoverk} If $K\not\subset k$ and at least one of $\ell, \frakf\cdot\Delta_K$ is odd then 
Lemma~\ref{lem:denom} and Proposition~\ref{prop:VAVcorrected} give 
\[\frac{\End_{\Gamma_k}(E[{\ell^n}])}{(\End(\overline{E})\otimes \Z/\ell^n\Z)^{\Gamma_k}}\cong \Z/s_{E,\ell^n}\Z \times ( \Z/t_{E,\ell^n}\Z)^2.\]
By \eqref{eq:geomBrauer}, $\Br(E\times E)/\Br_1(E\times E)$ is an abelian group of rank at most $2$. Therefore, $t_{E,\ell^n}=1$. It remains to bound $s_{E,\ell^n}$ for $n$ large. By Corollary~\ref{cor:stcomparison} it suffices to show that, for all $r\in\Z_{\geq 0}$, we have $s_{E',\ell^r}\leq \ell^m$. This follows from Lemma~\ref{lem:denom}, Proposition~\ref{prop:VAVcorrected}, Theorem~\ref{thm:rn} and Remark~\ref{rmk:m>0} applied to $E'$. 

From now on, we assume that $K\not\subset k$, $\ell=2$ and $2\mid \frakf\cdot\Delta_K$. So
Lemma~\ref{lem:denom} and Proposition~\ref{prop:VAVcorrected} give 
\[\frac{\End_{\Gamma_k}(E[{2^n}])}{(\End(\overline{E})\otimes \Z/2^n\Z)^{\Gamma_k}}\cong \frac{\Z/s_{E,2^n\Z}\times ( \Z/t_{E,2^n}\Z)^2}{\Z/2\Z}.\]
Since $\Br(E\times E)/\Br_1(E\times E)$ has rank at most $2$ we find that $t_{E,2^n}\leq 2$.

First suppose that $E[2]\neq E[2](k)$. Then $t_{E,2^n}=1$ for all $n\geq 0$ and hence
\[\frac{\End_{\Gamma_k}(E[{2^n}])}{(\End(\overline{E})\otimes \Z/2^n\Z)^{\Gamma_k}}\cong \frac{\Z/s_{E,2^n}\Z}{\Z/2\Z}\cong\frac{\Z}{(s_{E,2^n}/2)\Z} .\]
Now Lemma~\ref{lem:denom}, Proposition~\ref{prop:VAVcorrected} and Theorem~\ref{thm:rn} applied to $E'$ show that, for all $r\in\Z_{\geq 0}$, we have $s_{E',2^r}\leq 2^{m+1}$. Hence, 
Corollary~\ref{cor:stcomparison} shows that, for all $r\in\Z_{\geq 0}$, we have $s_{E,2^{r+\ord_2\frakf}}\leq 2^{m+1+\ord_2\frakf}$. Therefore, for large $n$ we have $s_{E,2^n}/2\leq 2^{m+\ord_2\frakf}$, as required.

Now suppose that $E[2]=E[2](k)$, i.e.~$k(E[2])=k$. Then, by definition of $t_{E, 2^n}$, we have $t_{E, 2^n}\geq 2$ for all $n\geq 1$. We already saw above that $t_{E, 2^n}\leq 2$. Hence $t_{E, 2^n}=2$ and 
\[\frac{\End_{\Gamma_k}(E[{2^n}])}{(\End( \overline{E})\otimes \Z/2^n\Z)^{\Gamma_k}}\cong  \frac{\Z/s_{E,2^n }\Z\times ( \Z/2\Z)^2}{\Z/2\Z}.\]
Again, since $\Br(E\times E)/\Br_1(E\times E)$ has rank at most $2$ we find that \[\frac{\End_{\Gamma_k}(E[{2^n}])}{(\End(\overline{E})\otimes \Z/2^n\Z)^{\Gamma_k}}\cong  \Z/s_{E,2^n}\Z\times  \Z/2\Z.\]

By Corollary~\ref{cor:stcomparison}, for all $r\in\Z_{\geq 0}$, we have $s_{E,2^{r+\ord_2\frakf}}\leq 2^{\ord_2\frakf}\cdot s_{E',2^r}$.  For large $n$, it follows from Lemma~\ref{lem:denom}, Proposition~\ref{prop:VAVcorrected} and Theorem~\ref{thm:rn} applied to $E'$ that $s_{E',2^n}\leq 2^{m+1}$.
When performing this calculation, one observes that the upper bound on $s_{E',2^n}$ can only be achieved if $2\mid \Delta_K$ and $t_{E',2^n}=1$. The latter condition is equivalent to $E'[2]\neq E'[2](k)$.
\end{proof}

\begin{cor}\label{cor:EEordertransc2}
Let $E$ be an elliptic curve over a number field $k$ with CM by an order of conductor $\mathfrak{f}$ in an imaginary quadratic field $K$. Then
\begin{equation}\label{eq:divstate}
\#\frac{\Br(E\times E)}{\Br_1(E\times E)}\mid 2\cdot \mathfrak{f}^2\cdot [k:\Q]^4\cdot\prod_{\substack{\ell\textup{ prime}, \  \ell\nmid [k:\Q]\\ \left(\ell-\left(\frac{\Delta_K}{\ell}\right)\right)\big| [\calO_K^\times:\calO_{\ell}^\times]\cdot [k:\Q] }} \ell^2.
\end{equation}
If $[k:\Q]\geq 2$ then
  \begin{equation}\label{eq:ub}
  \# \frac{\Br (E \times E)}{\Br_1(E \times E)} \leq \mathfrak{f}^2\cdot [k:\Q]^4.
  \end{equation}
\end{cor}

\begin{remark}
Note that  $ [\calO_K^\times:\calO_{\ell}^\times]$ divides $6$. Therefore $6$ could be used in place of $[\calO_K^\times:\calO_{\ell}^\times]$ in \eqref{eq:divstate}.
\end{remark}

\begin{proof}[Proof of Corollary \ref{cor:EEordertransc2}]
We begin by proving \eqref{eq:ub}.
Let $\varphi: E\to E'$ be an isogeny of degree $\mathfrak{f}$ to an elliptic curve with CM by $\calO_K$ as in Lemma~\ref{isogenoustomaximal}.
Let $m_\ell:=m_\ell(E')$ and let $\mathfrak{c}:= \prod_{\ell \textup{ prime}} \ell^{ m_\ell}$.  
 First we consider the case where $K\subset k$. Then taking a product over all primes in Theorem~\ref{thm:Alexis} gives
\begin{equation}\label{eq:Kink}
\# \frac{\Br (E \times E)}{\Br_1(E \times E)}\mid \frakf^2\cdot\frakc^2.
\end{equation}
By Lemma~\ref{lem:Kf} we have $K_{\frakc}\subset kK=k$ and hence $2\cdot[K_\frakc:K]\leq [k:\Q]$. Now Proposition~\ref{prop:aux_cond} gives 
\[\frakc\leq 3\cdot[K_\frakc:K]^2< [k:\Q]^2,\]
which gives the desired upper bound in this case.

Now we assume that $K\not\subset k$.  Taking a product over all primes in Theorem~\ref{thm:Alexis} gives
\begin{equation}\label{eq:Knotink}
\# \frac{\Br (E \times E)}{\Br_1(E \times E)}\mid 4	\cdot\frakf \cdot \frakc.
\end{equation}
For $K\notin\{\Q(i),\Q(\zeta_3)\}$, Proposition~\ref{prop:aux_cond} gives 
\begin{equation}\label{eq:deg2}
\frakc\leq [k:\Q]^2,
\end{equation}
since $K_{\frakc}\subset kK$ by Lemma~\ref{lem:Kf}. (Note that \eqref{eq:deg2} holds for $K=\Q(\sqrt{-7})$ because $[k:\Q]\geq 2$.) Now the desired upper bound follows by noting that $4\cdot\frakf\cdot[k:\Q]^2\leq\frakf\cdot[k:\Q]^4$ when $[k:\Q]\geq 2$.
For $K=\Q(\zeta_3)$, Theorem~\ref{thm:Alexis} and Remark~\ref{rmk:implications} yield
\begin{equation}\label{eq:z3}
\# \frac{\Br (E \times E)}{\Br_1(E \times E)}\mid \frakf^2\cdot\frakc
\end{equation}
and Proposition~\ref{prop:aux_cond} gives $\frakc\leq 3\cdot[K_\frakc:K]^2\leq 3\cdot[k:\Q]^2<[k:\Q]^4$.  For $K=\Q(i)$, Proposition~\ref{prop:aux_cond} shows that the only possible value of $\frakc$ violating the desired bound $\frakc\leq [k:\Q]^2$ is $\frakc=5$. But if $\frakc=5$ then Theorem~\ref{thm:Alexis} and Remark~\ref{rmk:implications} give
\begin{equation}
\# \frac{\Br (E \times E)}{\Br_1(E \times E)}\mid 2\cdot \frakf^2\cdot\frakc=10\cdot\frakf^2<\frakf^2\cdot[k:\Q]^4.
\end{equation}
This completes the proof of \eqref{eq:ub}.

For the divisibility statement \eqref{eq:divstate}, we first claim that if $m_\ell\geq 1$ then $\ell^{m_\ell-1}\mid [k:\Q]$. We have $K_{\ell^{m_\ell}}\subset K_\mathfrak{c}\subset kK$ so $[K_{\ell^{m_\ell}}:K]\mid [k:\Q]$. By \eqref{eq:form}, \begin{eqnarray*}
[K_{\ell^{m_\ell}}:K] &=& \ell^{m_\ell-1}\cdot\frac{h_K}{[\calO_K^\times:\calO_{\ell^{m_\ell}}^\times]}\cdot \left( \ell - \left( \frac{\Delta_K}{\ell} \right)\right)\\
&=& \ell^{m_\ell-1}\cdot [K_{\ell}:K],
\end{eqnarray*}
because $\calO_{\ell^{n}}^\times=\{\pm 1\}$ for all $n\geq 1$. Thus, $\ell^{m_\ell-1}\mid [K_{\ell^{m_\ell}}:K]$, proving the claim.

Now if $m_\ell \geq 2$ then $\ell^{m_\ell}\mid \ell^{2(m_\ell-1)}\mid  [k:\Q]^2$. It remains to deal with the primes $\ell$ for which $m_\ell=1$. By \eqref{eq:form} we have
\begin{eqnarray*}
[K_{\ell}:K]\cdot [\calO_K^\times:\calO_{\ell}^\times] =h_K\cdot \left( \ell - \left( \frac{\Delta_K}{\ell} \right)\right).
\end{eqnarray*}
Therefore, \[\frakc\mid [k:\Q]^2\cdot\prod_{\substack{\ell\textup{ prime}, \  \ell\nmid [k:\Q]\\ \left(\ell-\left(\frac{\Delta_K}{\ell}\right)\right)\big| [\calO_K^\times:\calO_{\ell}^\times]\cdot [k:\Q] }} \ell.\]
Now observe that in all cases Theorem~\ref{thm:Alexis} and Remark~\ref{rmk:implications} imply that
\begin{equation}\label{eq:div}
\# \frac{\Br (E \times E)}{\Br_1(E \times E)}\mid 2\cdot \frakf^2\cdot\frakc^2
\end{equation}
to complete the proof of \eqref{eq:divstate}.
\end{proof}

\begin{rmk}\label{rem:geomBr}  Similar results can be obtained for $\#\Br(\overline{E} \times \overline{E})^{\Gamma_k}$. For example, if $E$ has CM by $\calO_K$ then~\cite[Theorems~2.6 and~2.8]{Newton15}, Lemma~\ref{lem:Kf} and Proposition~\ref{prop:aux_cond} show that 
\[ \#\Br(\overline{E} \times \overline{E})^{\Gamma_k} =  |\Delta_K| \cdot \prod_{\ell \textup{ prime}} \ell^{2\cdot m_\ell(E)}\leq 3^2\cdot|\Delta_K|\cdot [k:\Q]^4.  \]
\end{rmk}

\begin{cor}\label{cor:EEordertranscunif2} 
Let $E$ be an elliptic curve over a number field $k$ with CM by an order in an imaginary quadratic field $K$. Then, for $k=\Q$, 
\begin{equation}\label{eq:Q}
\# \frac{\Br (E \times E)}{\Br_1(E \times E)} \leq \begin{cases} 4 &  \textrm{if } K=\Q(\sqrt{-7});\\
8 &\textrm{if } K=\Q(i);\\
9&\textrm{if } K=\Q(\zeta_3);\\
1&\textrm{otherwise.}
\end{cases}
\end{equation}
For $[k:\Q]\geq 2$,
\begin{equation}\label{eq:ub2}
\# \frac{\Br (E \times E)}{\Br_1(E \times E)} \leq [k:\Q]^8.
\end{equation}
In all cases, 
\begin{equation}\label{eq:divstate2}
\#\frac{\Br(E\times E)}{\Br_1(E\times E)}\mid 2\cdot [k:\Q]^8\cdot\prod_{\substack{\ell\textup{ prime}, \  \ell\nmid [k:\Q]\\ \left(\ell-\left(\frac{\Delta_K}{\ell}\right)\right)\big| [\calO_K^\times:\calO_{\ell}^\times]\cdot [k:\Q] }} \ell^4.
\end{equation}
\end{cor}

\begin{proof}
Let $\End(\overline{E})=\calO_\frakf$ and let $\frakc$ be as in the proof of Corollary~\ref{cor:EEordertransc2}.
To obtain \eqref{eq:ub2} and \eqref{eq:divstate2}, repeat the proof of Corollary~\ref{cor:EEordertransc2} noting that at each stage the bounds given for $\frakc$ also apply to $\frakf$ by Corollary~\ref{cor:conductor}. 
We finish by proving \eqref{eq:Q}. By Lemma~\ref{lem:Kf} and Corollary~\ref{cor:conductor}, we have $K_{\frakc}=K_\frakf=K$; we will use this in our applications of Proposition~\ref{prop:aux_cond}. If $K\notin\{\Q(\sqrt{-7}), \Q(\zeta_3),\Q(i)\}$ then Proposition~\ref{prop:aux_cond} shows that $\frakc=\frakf=1$. If $\frakf=1$ then the result follows from \cite[Theorem~3.1]{Newton15}, Theorem~\ref{thm:rn} and Remark~\ref{rmk:m>0}. Henceforth, suppose that $\frakf>1$ and $K\in \{\Q(\sqrt{-7}), \Q(\zeta_3),\Q(i)\}$. 

If $K=\Q(\sqrt{-7})$ then Proposition~\ref{prop:aux_cond} shows that $\frakc,\frakf\leq 2$. Thus, the result follows from Theorem~\ref{thm:Alexis} if we can show that any elliptic curve $E/\Q$ with $\End(\overline{E})=\Z[\sqrt{-7}]$ satisfies $E[2]\neq E[2](\Q)$. Up to a quadratic twist (which does not change the Galois module structure of the $2$-torsion), we may assume that $E$ is the elliptic curve \cite[\href{www.lmfdb.org/EllipticCurve/Q/49/a/1)}{Elliptic Curve 49.a.1}]{lmfdb}, which has Mordell--Weil group $\Z/2\Z$.

If $K=\Q(\zeta_3)$ then Proposition~\ref{prop:aux_cond} shows that $\frakc,\frakf\leq 3$. For $\frakf=3$, the result follows directly from Theorem~\ref{thm:Alexis} and Remark~\ref{rmk:implications} since $2\nmid \Delta_K$. For $\frakf=2$ we must verify that any elliptic curve $E/\Q$ with $\End(\overline{E})=\Z[\sqrt{-3}]$ satisfies $E[2]\neq E[2](\Q)$. As above, we need only check this for one specific curve since $\Aut\overline{E}=\{\pm 1\}$ and hence any twist of $E/\Q$ is a quadratic twist. We can take $E$ to be \cite[\href{www.lmfdb.org/EllipticCurve/Q/36/a/1)}{Elliptic Curve 36.a.1}]{lmfdb} which has Mordell--Weil group $\Z/2\Z$, for example.

If $K=\Q(i)$ then \eqref{eq:form} shows that $\frakc,\frakf\leq 2$. By the same reasoning as above, the result follows from the fact that the elliptic curve \cite[\href{www.lmfdb.org/EllipticCurve/Q/32/a/1)}{Elliptic Curve 32.a.1}]{lmfdb} has Mordell--Weil group $\Z/2\Z$.
\end{proof}

\section{The transcendental part of the Brauer group of a product of CM elliptic curves}\label{sec:isogCM}

In this section we give uniform bounds on the transcendental part of the Brauer group of a product $E_1\times E_2$ of CM elliptic curves. The curves may or may not be isogenous -- we deal with these two cases separately. In the case where $E_1$ and $E_2$ are isogenous we will use the isogeny to reduce to the case where $E_1=E_2$, which was dealt with in the previous section. We begin by bounding the difference in size of the transcendental parts of the Brauer groups of isogenous abelian varieties in terms of the degree of the isogeny.

\begin{proposition} \label{BRAdegisog} Let $A$ and $A'$ be abelian varieties of dimension $g$ over a field $k$ of characteristic $0$. Suppose that there exists a $k$-isogeny $\phi: A \to A'$ of degree $d$. Then the kernel of the induced map $\phi^*:\Br \overline{A'} \to \Br \overline{A}$ is contained in $\Br \overline{A'}[d]$. Consequently,
\[\#(\Br\overline{A'})^{\Gamma_k} \mid d^{g(2g-1)-\rho}\cdot\#(\Br\overline{A})^{\Gamma_k}\]
and
\[\#\frac{\Br A' }{\Br_1 A' }\mid d^{g(2g-1)-\rho}\cdot\#\frac{\Br A }{\Br_1 A },\]
where $\rho$ is the rank of $\NS\overline{A'}$ and we have $1\leq \rho\leq g^2$.
\end{proposition}

\begin{proof} The isogeny $\phi$ induces an injection of function fields \[\phi^*:\overline{k}(\overline{A'})\hookrightarrow \overline{k}(\overline{A})\] such that $[\overline{k}(\overline{A}):\phi^*(\overline{k}(\overline{A'}))]=d$. The map $\phi^*:\Br \overline{A'} \to \Br \overline{A}$ coincides with the restriction map $\Res_{\phi} : \Br \overline{k}(\overline{A'} )\to \Br \overline{k}(\overline{A} )$. Since $\Cor_\phi\circ\Res_\phi=[d]$, the kernel of $\phi^*:\Br \overline{A'} \to \Br \overline{A}$ is contained in $\Br \overline{A'} [d]$. The proof of \cite[Lemma~4.2]{BN} shows that $\#\Br \overline{A'} [d]= d^{g(2g-1)-\rho}$. The fact that $1\leq \rho\leq g^2$ is well known.

To complete the proof, recall that for any abelian variety $B$, there is an injection $\Br B/\Br_1 B\hookrightarrow (\Br\overline{B})^{\Gamma_k}$ by definition of $\Br_1 B$, and $(\Br\overline{B})^{\Gamma_k} $ is finite by \cite[Theorem~1.1]{SZ-K3finite}. The kernels of the induced maps $\phi^*: (\Br\overline{A'})^{\Gamma_k} \to  (\Br\overline{A})^{\Gamma_k} $ and $\phi^*: \Br A' /\Br_1 A' \to \Br A/\Br_1 A$ are contained in the kernel of $\phi^*:\Br \overline{A'} \to \Br \overline{A}$.
\end{proof}

Next, we bound the degree of an isogeny between CM elliptic curves in terms of the CM data.

\begin{proposition}\label{prop:boundCMisog}
Let $E_1$ and $E_2$ be elliptic curves over $\CC$ with complex multiplication by an order $\calO$ of conductor $\mathfrak{f}$ in an imaginary quadratic field $K$.
Then there is an isogeny  \mbox{$\varphi: E_{1}\to E_{2}$} such that 
\[\deg\varphi \leq 2\cdot\pi^{-1}\cdot\mathfrak{f}\cdot\sqrt{ |\Delta_K|}.\]
\end{proposition}

\begin{proof}
First note that all elliptic curves over $\CC$ with CM by $\calO$ are isogenous and that, up to isomorphism, any isogeny between elliptic curves over $\CC$ with CM by $\calO$ is of the form $\phi_\mathfrak{a}:E_\mathfrak{b}\to E_{\mathfrak{a}^{-1}\mathfrak{b}}$ for invertible $\calO$-ideals $\mathfrak{a}$ and $\mathfrak{b}$. Here $E_\mathfrak{b}$ corresponds to $\CC/\mathfrak{b}$ and $\phi_\mathfrak{a}$ is the natural map coming from the inclusion of lattices $\mathfrak{b}\subset \mathfrak{a}^{-1}\mathfrak{b}$. See \cite[Corollary~10.20]{Cox}, for example. We have $\deg\phi_\mathfrak{a}=N(\mathfrak{a})=[\calO:\mathfrak{a}]$ by \cite[Lemma~11.26]{Cox}, for example. Note that replacing $\mathfrak{a}$ by $\lambda\mathfrak{a}$ for $\lambda\in K^\times$ corresponds to replacing $ \mathfrak{a}^{-1}\mathfrak{b}$ by a homothetic lattice and hence does not change the isomorphism class of $E_{\mathfrak{a}^{-1}\mathfrak{b}}$. 
A simple application of Minkowski's theorem shows that there exists an $\calO$-ideal $\mathfrak{c}$ in the same ideal class as $\mathfrak{a}$ such that $N(\mathfrak{c})\leq 2\cdot\pi^{-1}\cdot\mathfrak{f}\cdot \sqrt{|\Delta_K|}$, since $\mathfrak{f}^2\cdot |\Delta_K|$ is the absolute value of the discriminant of $\calO$. Therefore $\phi_\mathfrak{c}$ is a suitable isogeny.
\end{proof}

\begin{cor}\label{cor:boundCMisog} Let $E_1$ and $E_2$ be elliptic curves over $\CC$ with complex multiplication by orders with conductors $\mathfrak{f}_1$ and $\mathfrak{f}_2$, respectively, in an imaginary quadratic field $K$. 
Then there is an isogeny $\varphi: E_{1}\to E_{2}$ such that 
\[\deg\varphi \leq 2\cdot\pi^{-1}\cdot\mathfrak{f}_1\cdot\mathfrak{f}_2 \cdot \sqrt{|\Delta_K|}.\]
\end{cor}

\begin{proof}
Lemma~\ref{isogenoustomaximal} shows the existence of isogenies $\phi_1:E_1\to E'_1$ and $\phi_2:E_2\to E'_2$ with degrees $\mathfrak{f}_1$ and $\mathfrak{f}_2$, respectively, where $E'_1$ and $E'_2$ have CM by $\calO_K$. Proposition~\ref{prop:boundCMisog} shows the existence of an isogeny $\varphi : E'_{1} \to E'_{2}$ such that 
$\deg\varphi \leq 2\cdot\pi^{-1}\cdot \sqrt{|\Delta_K|}$.  Let $\hat{\phi}_2: E'_{2}\to E_{2}$ be the dual of $\phi_2$. Now the isogeny $\hat{\phi}_2\circ \varphi\circ\phi_1:E_{1} \to E_{2}$ has degree at most $2\cdot\pi^{-1}\cdot\mathfrak{f}_1\cdot\mathfrak{f}_2\cdot \sqrt{|\Delta_K|}$, as desired.
\end{proof}

Now we combine the results obtained so far to obtain bounds for the transcendental parts of Brauer groups of products of CM elliptic curves. At several points we use the fact that for a variety $X/k$ and a finite extension $L/k$ we have
\begin{equation}\label{eq:basechange}
\frac{\Br X}{\Br_1 X} \hookrightarrow \frac{\Br X_L}{\Br_1 X_L}. 
\end{equation}

\begin{thm}\label{thm:boundCMisog1}
Let $k$ be a number field and let $E_1$ and $E_2$ be elliptic curves over $k$ with complex multiplication by orders with conductors $\mathfrak{f}_1$ and $\mathfrak{f}_2$, respectively, in an imaginary quadratic field $K$. Let $M/Kk$ be a finite extension such that all isogenies $\overline{E}_2\to\overline{E}_1$ are induced by isogenies defined over $M$. Then
\begin{eqnarray*} 
\# \frac{\Br(E_1 \times E_2)}{\Br_1(E_1 \times E_2)} &\leq & 2^2\cdot\pi^{-2}\cdot \mathfrak{f}_1^2\cdot\mathfrak{f}_2^2 \cdot |\Delta_K|\cdot[M:\Q]^4.
\end{eqnarray*}

Furthermore, if the class number of $K$ is $1$ then 
\[ \# \frac{\Br(E_1 \times E_2)}{\Br_1(E_1 \times E_2)} \leq\mathfrak{f}_1^2\cdot\mathfrak{f}_2^2\cdot[M:\Q]^4.\]

In all cases we can choose $M$ such that $[M:Kk]\mid \#\calO_K^\times$.
\end{thm}

\begin{proof}  
Let $\varphi: E_{2,M}\to E_{1,M}$ be an $M$-isogeny.
For $i=1,2$ let $\psi_i: E_i\to E'_i$ be the $k$-isogeny of degree $\mathfrak{f}_i$ to an elliptic curve over $k$ with CM by $\calO_K$ provided by Lemma~\ref{isogenoustomaximal}. Then $\psi_{1,M}\circ\varphi\circ\psi_{2,M}^\vee:E'_{2,M}\to E'_{1,M}$ is an $M$-isogeny. Thus, by Lemma~\ref{lem:allisog}, all isogenies $\overline{E'}_{2}\to \overline{E'}_{1}$ are defined over $M$. Let $\theta: E'_{2, M}\to E'_{1, M}$ be an isogeny of minimal degree. By Proposition~\ref{prop:boundCMisog}, 
\begin{equation}\label{minkbound}\deg{{\theta} }\leq   2\cdot\pi^{-1} \cdot\sqrt{|\Delta_K|}.\end{equation}
Now $(\id,\theta)\circ(\psi_{1,M},\psi_{2,M}): E_{1,M}\times E_{2,M}\to E'_{1,M}\times E'_{1,M}$ is an $M$-isogeny of degree $\mathfrak{f}_1\cdot\mathfrak{f}_2\cdot \deg\theta$. Now by \eqref{eq:basechange} and Proposition~\ref{BRAdegisog},
\[  \# \frac{\Br(E_{1}\times E_{2})}{\Br_1(E_{1}\times E_{2})} \mid  \# \frac{\Br(E_{1,M}\times E_{2,M})}{\Br_1(E_{1,M}\times E_{2,M})} \mid (\mathfrak{f}_1\cdot\mathfrak{f}_2 \cdot\deg\theta)^2 \cdot \# \frac{\Br(E'_{1,M}\times E'_{1,M})}{\Br_1(E'_{1,M}\times E'_{1,M})}.\]
Recall that $K\subset M$ and hence $[M:\Q]\geq 2$, whereby Corollary \ref{cor:EEordertransc2} gives 
\[\# \frac{\Br(E'_{1,M}\times E'_{1,M})}{\Br_1(E'_{1,M}\times E'_{1,M})} \leq [M:\Q]^4.\]
 Putting everything together yields the desired result. If the class number of $K$ is $1$ then all elliptic curves with CM by $\calO_K$ are isomorphic over $\overline{k}$ and hence $\deg\theta=1$.
 
Finally, by \cite[Proposition~1.3]{Remond} all isogenies $\overline{E}_2\to\overline{E}_1$ are induced by isogenies defined over a Galois extension of $Kk$ with degree dividing $\#\calO_K^\times$.
\end{proof}

Under the assumption of the Generalised Riemann Hypothesis, the following result gives a bound that only depends on the degree of the base field.

\begin{thm}\label{thm:boundCMisog2} Suppose that the Generalised Riemann Hypothesis holds. Let $k$ be a number field, let $K$ be an imaginary quadratic field,  and let $E_1, E_2$ be elliptic curves over $k$, each with (not necessarily full) CM by $K$. 
Let $M/Kk$ be a finite extension such that all isogenies $\overline{E}_2\to\overline{E}_1$ are induced by isogenies defined over $M$. Then $\# \frac{\Br(E_1 \times E_2)}{\Br_1(E_1 \times E_2)}$ is at most
\begin{align*}
 (3.4)^2\cdot 10^{8}\cdot[M:k]^4\cdot [k:\Q]^{12}\cdot\bigl((3.23)\cdot \log ([k:\Q]) +(2.73)\cdot 109\bigr)^4.
\end{align*}

 Moreover, we can choose $M$ such that $[M:Kk]\mid \#\calO_K^\times$.

\end{thm}

\begin{proof}
By \cite[Proposition~1.3]{Remond} and \cite[Th\'{e}or\`{e}me~1.4]{GR14}, there exists a number field $M$ with $[M:Kk]\mid \#\calO_K^\times$ and an $M$-isogeny $\varphi: E_{1,M}\to E_{2,M}$ with
\begin{equation*}
\deg\varphi\leq (3.4)\cdot 10^4\cdot [k:\Q]^2 \cdot \max\left\{h_F(E_1)+\frac{1}{2}\cdot \log ([k:\Q]),1\right\}^2,
\end{equation*} 
where $h_F$ is the stable Faltings height.
Under the assumption of the Generalised Riemann Hypothesis, \cite[Corollary 2.18]{Winckler} gives $h_F(E_1) \leq  (2.73) \cdot (109 + \log([k:\Q]))$ and hence
\begin{equation}\label{eq:degisog}
\deg\varphi\leq (3.4)\cdot 10^4\cdot [k:\Q]^2  \cdot\bigl((3.23)\cdot \log ([k:\Q]) +(2.73)\cdot 109\bigr)^2.
\end{equation} 
By \eqref{eq:basechange} and Proposition~\ref{BRAdegisog},
\begin{equation}\label{eq:Brisog}
  \# \frac{\Br(E_{1}\times E_{2})}{\Br_1(E_{1}\times E_{2})} \mid  \# \frac{\Br(E_{1,M}\times E_{2,M})}{\Br_1(E_{1,M}\times E_{2,M})} \mid (\deg\varphi)^2\cdot  \# \frac{\Br(E_{1,M}\times E_{1,M})}{\Br_1(E_{1,M}\times E_{1,M})}.
  \end{equation}
Since $K\subset M$, we have $[M:\Q]\geq 2$ whereby Corollary \ref{cor:EEordertransc2} gives 
\[\# \frac{\Br(E_{1,M}\times E_{1,M})}{\Br_1(E_{1,M}\times E_{1,M})} \leq \mathfrak{f}_1^2\cdot [M:\Q]^4,
\]
where $E_1$ has CM by the order of conductor $\mathfrak{f}_1$ in $K$. Combining this with \eqref{eq:Brisog} yields
\begin{equation}\label{eq:Brbound}
  \# \frac{\Br(E_{1}\times E_{2})}{\Br_1(E_{1}\times E_{2})}\leq  (\deg\varphi)^2\cdot \mathfrak{f}_1^2\cdot [M:\Q]^4.
\end{equation}
If $K\notin \{\Q(\sqrt{-7}), \Q(i),\Q(\zeta_3)\}$ then Corollary~\ref{cor:conductor} gives $\mathfrak{f}_1\leq [k:\Q]^2$, whereby the result follows from \eqref{eq:degisog} and \eqref{eq:Brbound}. On the other hand, if $K\in \{\Q(\sqrt{-7}), \Q(i),\Q(\zeta_3)\}$ then the result follows from Theorem~\ref{thm:boundCMisog1} and Corollary~\ref{cor:conductor}.
\end{proof}

\begin{thm} \label{nonisogfullCM} Suppose that the Generalised Riemann Hypothesis holds. Let $E_1$ and $E_2$ be geometrically non-isogenous elliptic curves over a number field $k$ such that $E_i$ has (not necessarily full) CM by an imaginary quadratic field $K_i$.
Then $ \#\frac{\Br(E_1 \times E_2)}{\Br_1(E_1 \times E_2)}$ is at most
\[2^{316}\cdot (241)^{24}\cdot [kK_1K_2:\Q]^{24}\cdot\bigl((5.46)\cdot(109+\log ([k:\Q]))+3\bigr)^{24}.\]
\end{thm}

\begin{proof} 
By the definition of the transcendental part of the Brauer group, we have
\[\frac{\Br(E_1 \times E_2)}{\Br_1(E_1 \times E_2)}\hookrightarrow \Br(\overline{E}_1 \times \overline{E}_2)^{\Gal(\overline{k}/kK_1K_2)}.\]
By \cite[Th\'{e}or\`{e}mes~1.5(3) and~1.8]{GR20}  the exponent of $\Br(\overline{E}_1 \times \overline{E}_2)^{\Gal(\overline{k}/K_1K_2k)}$ is at most $2\cdot d^{3/2}$, with
\begin{equation}\label{eq:Brnon1}
d\leq (241)^4\cdot 2^{52}\cdot [kK_1K_2:\Q]^4\cdot \max\{\log ([kK_1K_2:\Q]), h_F(E_1\times E_2)+3\}^4 ,
\end{equation}
where $h_F$ is the stable Faltings height, which satisfies $h_F(E_1\times E_2)=h_F(E_1)+h_F( E_2)$.  Under the assumption of the Generalised Riemann Hypothesis, \cite[Corollary 2.18]{Winckler} gives 
\begin{equation}\label{eq:Brnon2}
h_F(E_{i}) \leq  (2.73) \cdot (109 + \log([k:\Q])).
\end{equation}
By \eqref{eq:geomBrauer} we have $\Br(\overline{E}_1 \times \overline{E}_2)\cong (\Q/\Z)^4$ and hence 
\begin{equation}\label{eq:Brnon3}
\frac{\Br(E_1 \times E_2)}{\Br_1(E_1 \times E_2)}\leq (2\cdot d^{3/2})^4=2^4\cdot d^6.
\end{equation}
Combining \eqref{eq:Brnon1}--\eqref{eq:Brnon3} gives the desired result.
\end{proof}

\section{Uniform bound results for certain classes of abelian and K3 surfaces}
Let $k$ be a number field. In this section, we use the results obtained for products of CM elliptic curves in Section~\ref{sec:isogCM} alongside the results
of Sections~\ref{sec:ab} and \ref{sec:Kummer} to deduce bounds on the transcendental parts of the Brauer groups of singular abelian surfaces in $\scrA_k$ and certain elements of $\scrK_k$ related to products of CM elliptic curves. We begin with the results for abelian surfaces.

\begin{thm}\label{thm:ablatt}
Let $\Lambda$ be a rank $4$ lattice containing a hyperbolic plane, let $K:=\Q(\sqrt{\disc\Lambda})$, let $A \in \scrA_{k,\Lambda}$ and let $L/k$ be a finite extension such that $\End (A_L)=\End (\overline{A})$. Then

\[ \# \frac{\Br A}{\Br_1 A} \leq 2^{2}\cdot\pi^{-2}\cdot |\Delta_K|^{-1}\cdot |\disc \Lambda|^2 \cdot [L:\Q]^4 .\]

If $K$ has class number $1$ then

\[ \# \frac{\Br A}{\Br_1 A} \leq  |\Delta_K|^{-2}\cdot |\disc \Lambda|^2 \cdot [L:\Q]^4 .\]

In all cases we can choose $L$ such that $[L:k]\leq 2^4\cdot 3$. 
\end{thm}

\begin{proof}
By Lemma~\ref{lem:NSrks} and Proposition~\ref{prop:productsurface} there exist a finite extension $L/k$ with $[L:k]\leq 2^4\cdot 3$ and isogenous CM elliptic curves $E_1$ and $E_2$ over $L$ such that 
\[ A_L\cong E_1 \times E_2\] 
and $\End (A_L)=\End(\overline{A})$. Furthermore, Corollary~\ref{cor:NSExE} shows that the CM field is $K$ and 
\[\disc\Lambda=\disc \NS(\overline{E}_1\times \overline{E}_2)=\lcm(\mathfrak{f}_1,\mathfrak{f}_2)^2 \cdot \Delta_K\]
where $E_1$ and $E_2$ have CM by orders in $K$ of conductors $\mathfrak{f}_1$ and $\mathfrak{f}_2$, respectively. Proposition~\ref{prop:productsurface} shows that $K\subset L$. Since $\End(A_L)=\End(\overline{A})$, all isogenies between $E_1$ and $E_2$ are defined over $L$. Now the result follows from \eqref{eq:basechange} and Theorem~\ref{thm:boundCMisog1} applied to $A_L$. 
\end{proof}

Under the assumption of the Generalised Riemann Hypothesis, the next result gives a bound that only depends on $[k:\Q]$.

\begin{thm}\label{thmabvar}
Suppose that the Generalised Riemann Hypothesis holds. Let $A \in \scrA_{k}$ with $\rank\NS \Abar = 4$ and let $L/k$ be a finite extension such that $\End (A_L)=\End (\overline{A})$. Then
 \[ \# \frac{\Br A}{\Br_1 A} \leq(3.4)^2\cdot 10^{8}\cdot [L:\Q]^{12}\cdot\bigl((3.23)\cdot \log ([L:\Q]) +(2.73)\cdot 109\bigr)^4.\]
 Moreover, we can choose $L$ such that $[L:k]\leq 2^4\cdot 3$. 
\end{thm}

\begin{proof} By Lemma~\ref{lem:NSrks} and Proposition~\ref{prop:productsurface} there exist a finite extension $L/k$ with $[L:k]\leq 2^4\cdot 3$ and isogenous CM elliptic curves $E_1$ and $E_2$ over $L$ such that 
\[ A_L\cong E_1 \times E_2\] 
and $\End (A_L)=\End(\overline{A})$. 
By \eqref{eq:basechange} we have
\[\# \frac{\Br A}{\Br_1 A} \mid \# \frac{\Br A_L}{\Br_1 A_L}=\# \frac{\Br (E_1\times E_2)}{\Br_1 (E_1\times E_2)}. \]
Now apply Theorem~\ref{thm:boundCMisog2} to $E_1\times E_2$ over $L$, noting that since $\End (A_L)=\End(\overline{A})$ we can take $M=L$ in Theorem~\ref{thm:boundCMisog2}.
\end{proof}

Next we give our results for K3 surfaces related to products of CM elliptic curves. The bounds obtained depend on whether the elliptic curves are isogenous.

\begin{theorem} \label{thm:lattprodisog}
Let $\Lambda$ be the N\'eron--Severi lattice of the Kummer surface of a product of isogenous (not necessarily full) CM elliptic curves over $\overline{k}$, let $K:=\Q(\sqrt{\disc\Lambda})$, and let $X\in \scrK_{k,\Lambda}$. Then there exist a finite extension $L/k$ and elliptic curves $E_1,E_2$ over $L$ such that $X_L\cong \Kum (E_1\times E_2)$ and we have 
\[ \# \frac{\Br X}{\Br_1 X} \leq 2^{-2}\cdot\pi^{-2}\cdot |\Delta_K|^{-1}\cdot |\disc \Lambda|^2 \cdot [L:\Q]^4.\]

If $K$ has class number $1$ then
\[ \# \frac{\Br X}{\Br_1 X} \leq 2^{-4}\cdot |\Delta_K|^{-2}\cdot |\disc \Lambda|^2  \cdot [L:\Q]^4 .\]

In all cases we can choose $L$ such that $[L:k]\leq 2^9\cdot 3\cdot M(20)$.
\end{theorem}

\begin{proof}
By Proposition~\ref{prop:Nik}, Theorem~\ref{thm:KumE}, and Proposition~\ref{propVAV}, there exist a finite extension $L/k$ and isogenous (not necessarily full) CM elliptic curves $E_1,E_2$ over $L$ such that $X_{L}\cong \Kum(E_1\times E_2)$. Corollary~\ref{cor:NSKum} shows that the CM field is $K$ and
 \[|\disc\Lambda|=|\disc \NS(\Kum(\overline{E}_1\times \overline{E}_2))|=2^2\cdot \lcm(\mathfrak{f}_1,\mathfrak{f}_2)^2 \cdot |\Delta_K|,\] 
 where $E_1$ and $E_2$ have CM by orders in $K$ of conductors $\mathfrak{f}_1$ and $\mathfrak{f}_2$, respectively.
 By Remark~\ref{rem:M0}, Theorem~\ref{thm:KumE}, and Proposition~\ref{propVAV}, we have $K\subset L$ and
\[[L:k]\leq 2\cdot M(20)\cdot 2^4\cdot 3 \cdot 2^4=2^9\cdot 3\cdot M(20). \]
The result now follows from \eqref{eq:basechange}, \cite[Theorem~2.4]{SZ12} (cf.~Corollary~\ref{cor:Kumab}), and Theorem~\ref{thm:boundCMisog1}.
\end{proof}

\begin{theorem} \label{thm:unifsingK3_2}Suppose that the Generalised Riemann Hypothesis holds. 
Let $X\in\scrK_k$ be such that $\rank \NS\overline{X}=20$. Then there exists a finite extension $L/k$ such that $X_L\cong \Kum (E_1\times E_2)$ for some elliptic curves $E_1,E_2$ over $L$ and we have 
\[ \# \frac{\Br X}{\Br_1 X} \leq (3.4)^2\cdot 10^{8}\cdot [L:\Q]^{12}\cdot\bigl((3.23)\cdot \log ([L:\Q]) +(2.73)\cdot 109\bigr)^4.\]
Moreover, we can choose $L$ such that $[L:k]\leq 2^9\cdot 3\cdot M(20)$.
\end{theorem}

\begin{proof}
By Remark~\ref{rem:M0}, Theorem~\ref{thm:KumE} and Corollary~\ref{cor:Kumab}, there exists an extension $L/k$ with $[L:k]\leq 2^9\cdot 3\cdot M(20)$ and elliptic curves $E_1, E_2$ over $L$ such that 
\[\frac{\Br X}{\Br_1 X}\hookrightarrow \frac{\Br (E_1\times E_2)}{\Br_1( E_1\times E_2)}\]
and $\End(E_1\times E_2)=\End (\overline{E}_1\times \overline{E}_2)$. Now the result follows from Theorem~\ref{thmabvar}.
\end{proof}

\begin{theorem}\label{thm:singcover}Suppose that the Generalised Riemann Hypothesis holds. 
Let $X/k$ be a singular K3 surface, i.e.~a K3 surface with $\rank\NS\overline{X}=20$. Then
\[ \begin{array}{rl}
\# \frac{\Br X}{\Br_1 X} \leq & 2^{130}\cdot 3^{12}\cdot5^8\cdot (3.4)^2\cdot M(20)^{12}\cdot [k:\Q]^{12}\\
&\cdot \bigl((3.23)\cdot \log (2^{10}\cdot3\cdot M(20)\cdot [k:\Q]) +(2.73)\cdot 109\bigr)^4.
\end{array}\]
\end{theorem}

\begin{proof}
The proof of \cite[Theorem~1]{Shafarevich} shows there is a double cover $\varphi:Y \dasharrow X$ such that $Y$ and $\varphi$ are defined over an extension $k'/k$ of degree at most $2\cdot M(20)$ and $\overline{Y}$ is a Kummer surface with $\rank\NS\overline{Y}=20$.  Then Theorem~\ref{thm:unifsingK3_2} gives
\[\# \frac{\Br Y}{\Br_1 Y}\leq (3.4)^2\cdot 10^{8}\cdot [L:\Q]^{12}\cdot\bigl((3.23)\cdot \log ([L:\Q]) +(2.73)\cdot 109\bigr)^4\]
where $L$ is a finite extension of $k'$ built from the field extensions in Proposition~\ref{prop:Nik}, Theorem~\ref{thm:KumE} and Corollary~\ref{cor:Kumab}. Comparing the proofs of~\cite[Proposition~2.1]{VAV16} and~\cite[Theorem~1]{Shafarevich} shows that the field extension in Proposition~\ref{prop:Nik} is at most quadratic over $k'$, since $\Gamma_{k'}$ acts trivially on $\NS\overline{X}$ by construction. Consequently,
$[L:k']\leq 2^9\cdot 3$ and hence $[L:k]\leq 2^{10}\cdot 3\cdot M(20)$. 
The proof of \cite[Corollary~2.2]{IeronymouSkoro} shows that $\varphi$ induces a map $\varphi^*:\Br X_{k'}/\Br_1 X_{k'}\to \Br Y/\Br_1 Y$ whose kernel is killed by $2$.  Therefore, $\ker\varphi\hookrightarrow \Br \overline{X}[2]$ and, using \eqref{eq:basechange}, this yields
\[ \# \frac{\Br X}{\Br_1 X} \leq \#\Br \overline{X}[2] \cdot \# \frac{\Br Y}{\Br_1 Y}.\]
Now use that $\Br \overline{X}\cong(\Q/\Z)^2$, as follows from work of Grothendieck.
\end{proof}

\begin{thm} \label{unifKk2}Suppose that the Generalised Riemann Hypothesis holds. Let $X \in \scrK_{k}$ be geometrically isomorphic to the Kummer surface of the product of two non-isogenous (not necessarily full) CM elliptic curves over $\CC$. Then $\# \frac{\Br X}{\Br_1 X}$ is at most
\[ 2^{508}\cdot (241)^{24}\cdot M(18)^{24}\cdot [k:\Q]^{24}\cdot\bigl((5.46)\cdot(109+\log (2^6\cdot M(18)\cdot[k:\Q]))+3\bigr)^{24}.\]
\end{thm}

\begin{proof} 
By Remark~\ref{rem:M0}, Theorem~\ref{thm:KumE}, and Corollary~\ref{cor:Kumab}, there exist a finite extension $L/k$ with $[L:k]\leq 2^{6} \cdot M(18)$ and elliptic curves $E_1$ and $E_2$ over $L$ such that 
\[\frac{\Br X}{\Br_1 X}\hookrightarrow \frac{\Br (E_1\times E_2)}{\Br_1( E_1\times E_2)}.\]
Now apply Theorem~\ref{nonisogfullCM}.
\end{proof}

\begin{thm} \label{thm:effcomp}
Let $k$ be a number field and let $X/k$ be such that $\overline{X}$ is a Kummer surface with $\rank \NS\overline{X}=20$. Then $X(\A_k)^{\Br}$ is effectively computable. 
\end{thm}

\begin{proof} This follows from \cite[Theorem 1]{KT11}, \cite[Theorem 8.38]{PTvL}, Theorem~\ref{thm:lattprodisog} and Remark~\ref{unifBr1}.
\end{proof}

\begin{thm} \label{thm:effcompGRH}Suppose that the Generalised Riemann Hypothesis holds. 
Let $X/k$ be a singular K3 surface or a surface that is geometrically isomorphic to the Kummer surface of the product of elliptic curves $E_1$ and $E_2$ over $\CC$, where $E_i$ has CM by an order $\calO_i$ in a CM field $K_i$ for $i=1,2$. Then $X(\A_k)^{\Br}$ is effectively computable. 
\end{thm}

\begin{proof} This follows from \cite[Theorem 1]{KT11}, \cite[Theorem 8.38]{PTvL}, Theorems~\ref{thm:unifsingK3_2}, \ref{thm:singcover},~\ref{unifKk2}, and Remark~\ref{unifBr1}.
\end{proof}

\bibliography{bibshort}

\begin{bibdiv}
\begin{biblist}

\bib{BN}{article}{
      author={Balestrieri, F.},
      author={Newton, R.},
       title={Arithmetic of rational points and zero-cycles on products of
  {K}ummer varieties and {K3} surfaces},
        date={2019},
     journal={International Mathematics Research Notices},
      volume={rny303},
}

\bib{CasselsFrohlich}{book}{
      author={Cassels, J. W.~S.},
      author={Fr{\"o}hlich, A.},
       title={Algebraic number theory},
   publisher={London Mathematical Society},
        date={2010},
}

\bib{CFTTV16Published}{article}{
      author={Cantoral-Farf{\'a}n, V.},
      author={Tang, Y.},
      author={Tanimoto, S.},
      author={Visse, E.},
       title={Effective bounds for {B}rauer groups of {K}ummer surfaces over
  number fields},
        date={2018},
     journal={J. Lond. Math. Soc.},
      volume={97},
      number={2},
       pages={353\ndash 376},
}

\bib{Cox}{book}{
      author={Cox, D.~A.},
       title={Primes of the form $x^2 + ny^2$. {F}ermat, {C}lass {F}ield
  {T}heory, and {C}omplex {M}ultiplication},
      series={Monographs and textbooks in pure and applied mathematics},
   publisher={Wiley},
        date={1989},
}

\bib{DanielsLozanoRobledo}{article}{
      author={Daniels, H.~B.},
      author={Lozano-Robledo, \'{A}.},
       title={On the number of isomorphism classes of {CM} elliptic curves
  defined over a number field},
        date={2015},
     journal={Journal of {N}umber {T}heory},
      volume={157},
       pages={367\ndash 396},
}

\bib{FKRS}{article}{
      author={Fit\'{e}, F.},
      author={Kedlaya, K.},
      author={Rotger, V.},
      author={Sutherland, A.},
       title={Sato--{T}ate distributions and {G}alois endomorphism modules in
  genus 2},
        date={2012},
     journal={Compositio Mathematica},
      volume={148},
       pages={1390\ndash 1442},
}

\bib{GR14}{article}{
      author={Gaudron, {\'E}.},
      author={R{\'e}mond, G.},
       title={Th{\'e}or{\`e}me des p{\'e}riodes et degr{\'e}s minimaux
  d'isog{\'e}nies},
        date={2014},
     journal={Comment. Math. Helv.},
      volume={89},
      number={2},
       pages={343\ndash 403},
}

\bib{GR20}{article}{
      author={Gaudron, {\'E}.},
      author={R{\'e}mond, G.},
       title={Nouveaux th{\'e}or{\`e}mes d’isog{\'e}nie},
        note={To appear in M{\'e}m. Soc. Math. France},
}

\bib{Grothendieck}{article}{
      author={Grothendieck, A.},
       title={Le groupe de {B}rauer {II}},
        date={1968},
     journal={Dix expos{\'e}s sur la cohomologie des sch{\'e}mas},
       pages={46\ndash 188},
}

\bib{Huybrechts}{book}{
      author={Huybrechts, D.},
       title={Lectures on {K3} surfaces},
      series={Cambridge Studies in Advanced Mathematics},
   publisher={Cambridge University Press},
        date={2016},
      volume={158},
}

\bib{IeronymouSkoro}{article}{
      author={Ieronymou, E.},
      author={Skorobogatov, A.N.},
       title={Odd order {B}rauer--{M}anin obstruction on diagonal quartic
  surfaces},
        date={2015},
     journal={Advances in Mathematics},
      volume={270},
       pages={181\ndash 205},
        note={Corrigendum: Advances in Mathematics 307 (2017), 1372-1377},
}

\bib{Ito}{article}{
      author={Ito, K.},
       title={{On the Supersingular Reduction of K3 Surfaces with Complex
  Multiplication}},
        date={2018},
     journal={International Mathematics Research Notices},
        note={rny210},
}

\bib{Kani2011}{article}{
      author={Kani, E.},
       title={Products of {CM} elliptic curves},
        date={2011},
     journal={Collect. Math.},
      volume={62},
       pages={297\ndash 339},
}

\bib{Kani2016}{article}{
      author={Kani, E.},
       title={The moduli spaces of {J}acobians isomorphic to a product of two
  elliptic curves},
        date={2016},
     journal={Collect. Math.},
      volume={67},
      number={1},
       pages={21\ndash 54},
}

\bib{KT11}{article}{
      author={Kresch, A.},
      author={Tschinkel, Yu.},
       title={Effectivity of {B}rauer–{M}anin obstructions on surfaces},
        date={2011},
     journal={Advances in Mathematics},
      volume={226},
      number={5},
       pages={4131\ndash 4144},
}

\bib{Laface}{article}{
      author={Laface, R.},
       title={Decompositions of singular abelian surfaces},
        date={2019},
     journal={Asian Journal of Mathematics},
      volume={23},
      number={1},
       pages={157\ndash 172},
}

\bib{lmfdb}{misc}{
      author={{LMFDB Collaboration}, The},
       title={The {L}-functions and modular forms database},
         how={\url{http://www.lmfdb.org}},
        date={2019},
        note={[Online; accessed 30 October 2019]},
}

\bib{LP}{article}{
      author={Looijenga, E.},
      author={Peters, C.},
       title={Torelli theorems for{K}{\"a}hler {K3} surfaces},
        date={1980/81},
     journal={Compos. Math.},
      volume={42},
       pages={145\ndash 186},
}

\bib{Manin}{incollection}{
      author={Manin, Yu.~I.},
       title={Le groupe de {B}rauer--{G}rothendieck en g\'eom\'etrie
  diophantienne},
        date={1971},
   booktitle={Actes du {C}ongr\`es {I}nternational des {M}ath\'ematiciens
  ({N}ice, 1970), {T}ome 1},
   publisher={Gauthier-Villars, Paris},
       pages={401\ndash 411},
}

\bib{Minkowski}{article}{
      author={Minkowski, H.},
       title={Zur {T}heorie der positiven quadratischen {F}ormen},
        date={1887},
     journal={J. reine angew. Math.},
      volume={101},
       pages={196\ndash 202},
}

\bib{Newton15}{article}{
      author={Newton, R.},
       title={Transcendental {B}rauer groups of products of {CM} elliptic
  curves},
        date={2016},
     journal={J. Lond. Math. Soc.},
      volume={93},
      number={2},
       pages={397\ndash 419},
}

\bib{OrrSkoro}{article}{
      author={Orr, M.},
      author={Skorobogatov, A.~N.},
       title={Finiteness theorems for {K3} surfaces and abelian varieties of
  {CM} type},
        date={2018},
     journal={Compositio Mathematica},
      volume={154},
      number={8},
       pages={1571\ndash 1592},
}

\bib{PTvL}{article}{
      author={Poonen, B.},
      author={Testa, D.},
      author={van Luijk, R.},
       title={Computing {N}{\'e}ron--{S}everi groups and cycle class groups},
        date={2015},
     journal={Compositio Mathematica},
      volume={151},
      number={4},
       pages={713\ndash 734},
}

\bib{Remond}{article}{
      author={R{\'e}mond, G.},
       title={Degr{\'e} de d{\'e}finition des endomorphismes d'une
  vari{\'e}t{\'e} ab{\'e}lienne},
        date={2020},
     journal={J. Eur. Math. Soc.},
      volume={22},
       pages={3059\ndash 3099},
}

\bib{Schuett}{article}{
      author={Sch\"{u}tt, M.},
       title={{K3} surfaces with {P}icard rank 20},
        date={2010},
     journal={Algebra \& {N}umber {T}heory},
      volume={4},
      number={3},
       pages={335\ndash 356},
}

\bib{Serre}{incollection}{
      author={Serre, J.-P.},
       title={Bounds for the orders of the finite subgroups of {G}(k)},
        date={2007},
   booktitle={Group representation theory},
   publisher={EPFL Press},
       pages={405\ndash 450},
}

\bib{Shafarevich}{incollection}{
      author={Shafarevich, I.~R.},
       title={On the arithmetic of singular {K3}-surfaces},
        date={1996},
   booktitle={{A}lgebra and {A}nalysis ({K}azan, 1994)},
   publisher={de Gruyter},
       pages={103\ndash 108},
}

\bib{Shioda}{incollection}{
      author={Shioda, T.},
       title={Correspondence of elliptic curves and {M}ordell--{W}eil lattices
  of certain elliptic {K3}’s},
        date={2007},
   booktitle={Annual {EAGER} conference and workshop on the occasion of {J}.
  {M}urre's 75--th birthday on \emph{{A}lgebraic {C}ycles and {M}otives},
  {A}ugust 30--{S}eptember 3 2004, {L}eiden {U}niversity, the {N}etherlands},
   publisher={Cambridge University Press},
       pages={319\ndash 339},
}

\bib{SI}{incollection}{
      author={Shioda, T.},
      author={Inose, H.},
       title={On singular {K3} surfaces},
        date={1977},
   booktitle={Complex analysis and algebraic geometry},
      editor={Jr, W. L.~Baily},
      editor={Shioda, T.},
      series={Iwanami Shoten, Tokyo},
       pages={119\ndash 136},
}

\bib{SilvermanII}{book}{
      author={Silverman, J.~H.},
       title={Advanced topics in the arithmetic of elliptic curves},
      series={Graduate {T}exts in {M}athematics},
   publisher={Springer-Verlag},
        date={1994},
}

\bib{Skorobogatov-K3Conj}{article}{
      author={Skorobogatov, A.~N.},
       title={Diagonal quartic surfaces},
        date={2009},
     journal={Explicit Methods in Number Theory. Oberwolfach report},
      number={33},
       pages={76\ndash 79},
}

\bib{Sound}{article}{
      author={Soundararajan, K.},
       title={The number of imaginary quadratic fields with a given class
  number},
        date={2007},
     journal={Hardy--Ramanujan Journal},
      volume={30},
       pages={13\ndash 18},
}

\bib{SZ-K3finite}{article}{
      author={Skorobogatov, A.~N.},
      author={Zarhin, Yu.~G.},
       title={A finiteness theorem for the {B}rauer group of abelian varieties
  and {K}3 surfaces},
        date={2008},
     journal={J. Algebraic Geom.},
      volume={17},
      number={3},
       pages={481\ndash 502},
}

\bib{SZ12}{article}{
      author={Skorobogatov, A.~N.},
      author={Zarhin, Yu.~G.},
       title={The {B}rauer group of {K}ummer surfaces and torsion of elliptic
  curves},
        date={2012},
     journal={J. reine angew. Math.},
      number={666},
       pages={115\ndash 140},
}

\bib{SZ-KumVar}{article}{
      author={Skorobogatov, A.~N.},
      author={Zarhin, Yu.~G.},
       title={Kummer varieties and their {B}rauer groups},
        date={2017},
     journal={Pure Appl. Math. Quart.},
      volume={13},
       pages={337\ndash 368},
}

\bib{Tatuzawa}{article}{
      author={Tatuzawa, T.},
       title={On a theorem of {S}iegel},
        date={1951},
     journal={Japan. J. Math.},
      volume={21},
       pages={163\ndash 178},
}

\bib{Sage}{manual}{
      author={{The Sage Developers}},
       title={{S}agemath, the {S}age {M}athematics {S}oftware {S}ystem
  ({V}ersion 9.0)},
        date={2020},
        note={{\tt https://www.sagemath.org}},
}

\bib{VA-AWS}{incollection}{
      author={V{\'a}rilly-Alvarado, A.},
       title={Arithmetic of {K3} surfaces},
        date={2017},
   booktitle={Geometry over {N}onclosed {F}ields},
      series={Simons {S}ymposia},
   publisher={Springer},
       pages={197\ndash 248},
}

\bib{Valloni2}{misc}{
      author={Valloni, D.},
       title={The theory of complex multiplication for {K3} surfaces},
        date={2018},
        note={Preprint, arXiv:1804.08763v1},
}

\bib{Valloni}{article}{
      author={Valloni, D.},
       title={Complex multiplication and {B}rauer groups of {K3} surfaces},
        date={2021},
     journal={Advances in Mathematics},
      volume={385},
}

\bib{VAV16}{article}{
      author={V{\'a}rilly-Alvarado, A.},
      author={Viray, B.},
       title={Abelian $n$-division fields of elliptic curves and {B}rauer
  groups of product {K}ummer and abelian surfaces},
        date={2017},
     journal={Forum of Mathematics, Sigma},
      volume={5},
      number={E26},
}

\bib{VAVCorrigendum}{article}{
      author={V{\'a}rilly-Alvarado, A.},
      author={Viray, B.},
       title={Corrigendum to “{A}belian $n$-division fields of elliptic
  curves and {B}rauer groups of product {K}ummer and abelian surfaces"},
        date={2020},
        note={Appendix in arXiv: 1606.09240v5},
}

\bib{Winckler}{article}{
      author={Winckler, B.},
       title={Probl{\`e}me de {L}ehmer sur les courbes elliptiques {\`a}
  multiplications complexes},
        date={2018},
     journal={Acta Arith.},
      number={182},
       pages={347\ndash 396},
}

\end{biblist}
\end{bibdiv}

\end{document}